\documentclass{amsart}
\usepackage{wrapfig}
\usepackage[dvips]{graphicx}
\usepackage{amsmath,amsthm,amssymb,amscd}
\usepackage{braket}
\theoremstyle{definition}
\newtheorem{thm}{Theorem}[section]
\newtheorem{Def}[thm]{Definition}
\newtheorem{pro}[thm]{Proposition}
\newtheorem{cor}[thm]{Corollary}
\newtheorem{lem}[thm]{Lemma}
\newtheorem{ex}[thm]{Example}
\newtheorem{rem}[thm]{Remark}
\theoremstyle{definition}
\def\c#1{\mathcal {#1}}
\def\m#1{\mathbb {#1}}
\def\r#1{\rm {#1}}

\begin{document}

\title{Fundamental group of $C^*$-algebras with finite dimensional trace space}
\author{Takashi Kawahara}
\address[Takashi Kawahara]{Doctoral student of Mathematics, 
Kyushu University, Ito, 
Fukuoka, 819-0395,  Japan}
\email{t-kawahara@math.kyushu-u.ac.jp}      
\maketitle
\begin{abstract}
We introduce the fundamental group $\r{F}(\c{A})$ of 
a unital $C^*$-algebra $\c{A}$ with finite dimensional trace space.  
The elements of 
fundamental group are restricted by 
K-theoretical obstruction and positivity.
Moreover we show there are uncountably many mutually nonisomorphic simple $C^{*}$-algebras such that $\r{F}(\c{A})=\set{I_{n}}$.  
Our study is due to the results of 
fundamental group by Nawata and Watatani.

\end{abstract}

\section{Introduction}
  We recall some facts of fundamental groups of operator algebras.  
The fundamental group $\r{F}(M)$ of a $\r{II}_1$-factor $\r{M}$ 
with a normalized trace $\tau$ is defined by Murray and von Neumann in \cite{MN}.  
In the paper, the fact that if $M$ is  hyperfinite, then 
$\r{ F}(M) = {\mathbb R_+^{\times}}$ is shown. 
It was proved that $\r{ F}(L(\mathbb{F}_{\infty}))$ of the group factor 
of the free group $\mathbb{F}_{\infty}$ contains the positive rationals by Voiculescu in \cite{Vo} and 
it was shown that 
$\r{ F}(L(\mathbb{F}_{\infty})) = {\mathbb R}_+^{\times}$ by Radulescu in 
\cite{RF}.  The fact that $\r{ F}(L(G))$ is a countable group if $G$ is an ICC group with property (T) is shown by Connes \cite{Co}.  That either countable subgroup of $\mathbb R_+^{\times}$ or any uncountable group belonging to a certain "large" class 
can be realized as the fundamental group of some 
factor of type $\r{ II}_1$ is shown by Popa and Vaes in \cite{Po1} and in \cite{PoVa}.  

\ Nawata and Watatani \cite{NY}, \cite{NY2} introduce the fundamental group of simple $C^*$-algebras 
with unique trace.  Their study is essentially based on the computation  of 
Picard groups by Kodaka \cite{kod1}, \cite{kod2}, \cite{kod3}.  
Nawata defined the fundamental group of non-unital $C^{*}$-algebras \cite{N1} and calculate the Picard group of some projectionless $C^{*}$-algebras with strict comparison by the fundamental groups\cite{N2}.
In this paper, we define the fundamental group of $C^*$-algebras 
with finite dimensional trace space.  
This fundamental group is a "numerical invariant".  
Let $\c{A}$ be a unital $C^*$-algebra 
with finite dimensional bounded trace space.  
We define the fundamental group $\r{ F}(\c{A})$ of $\c{A}$ 
and the determinant fundamental group $\r{ F}_{\r{det}}(\c{A})$ by using self similarity and the extremal points of the bounded trace space of $\c{A}$.  
Then the groups $\r{ F}(\c{A})$ and $\r{ F}_{\r{det}}(\c{A})$ are 
multiplicative subgroups of $GL_{n}(\m{R})$ and ${\m{R}_+^{\times}}$ respectively.  
We shall show that the element of the fundamental group is restricted by K-theorical obstruction and positivity  and we will have that $A=DU(\sigma)$ for some diagonal matrix $D$ and for some permutation unitary $U(\sigma)$ for any $A$ in $\r{ F}(\c{A})$.  
If the unital $C^{*}$-algebras $\c{A}$ and $\c{B}$ with finite dimensional trace space are Morita equivalent, then $\r{ F}(\c{A})=(DU(\sigma))^{-1}\r{ F}(\c{B})(DU(\sigma))$ for some diagonal matrix $D$ and for some permutation unitary $U(\sigma)$.    
Moreover, we compute $\r{ F}(\c{A})$ of 
several $C^*$-algebras $A$. We shall show that given any group $G$ in ${\rm GL}_{2}(\m{R})$ which is isomorphic to $\m{Z}_{2}$ and whose elements have the form $DU(\sigma)$, there exists a simple $AF$-algebra $\c{A}$ such that $\r{ F}(\c{A})=G$.  
Furthermore, we shall show for any $n\in\m{N}$ there exisit uncountably many mutually nonisomorphic simple (non)nuclear unital $C^{*}$-algebras $\c{A}$ with $n$-dimensional trace space such that $\r{ F} (\c{A})=\set{I_{n}}$, where $\r{I}_{n}$ is a unit in $M_{n}(\m{C})$.  
\\
\\
\ We review some of the elementary facts on trace space.  
Let $\c{A}$ be a unital $C^{*}$-algebra.  
A trace $\varphi$ on $\c{A}$ is a linear functional on $\c{A}$ 
which satisfies $\varphi(ab)=\varphi(ba)$ for any $a,b$ in $\c{A}$.  
We denote by $\r{ T}(\c{A})$ the set of bounded traces on $\c{A}$ 
and denote by $\r{ T}(\c{A})^{+}_{1}$ the set of bounded positive traces whose norm is one.    
If $\r{T}(\c{A})\neq \set{0}$, then $\r{ T}(\c{A})^{+}_{1}$ is nonempty compact set of $\c{A}^{*}$ in the weak$^{*}$-topology 
so the set of extreme points of $\r{ T}(\c{A})^{+}_{1}$, denoted by $\partial _{e}\r{T}(\c{A})^{+}_{1}$, is nonempty.  
We suppose $\sharp(\partial_{e}{\rm T}(\c{A})^{+}_{1})=n$ for some $n$ in $\m{N}$ and say $\set{\varphi_{i}}^{n}_{i=1}=\partial_{e}{\rm T}(\c{A})^{+}_{1}$.  
Then $\set{\varphi_{i}}^{n}_{i=1}$ is a basis of $\r{ T}(\c{A})$.  

\section{Hilbert $C^*$-modules, imprimitivity bimodules and Picard groups}\label{sec:picard group} 
We recall some of the standard facts on imprimitivity bimodules.  
(See \cite{Lan}, \cite{MT}, \cite{Rie2}) 
Let ${\mathcal A}$ and ${\mathcal B}$ be unital $C^*$-algebras.  
The $dual\ module$ of an ${\mathcal A}$-${\mathcal B}$ imprimitivity bimodule ${\mathcal E}$, denoted by ${\mathcal E}^{*}$,  
is defined to be a set $\{\xi^*:\xi\in{\mathcal E}\}$ with the operations 
such that $\xi^*+\eta^*=(\xi+\eta)^*$, 
$\lambda\xi^*=(\overline{\lambda}\xi)^*,b\xi^* a
=(a^*\xi b^*)^*,{}_{{\mathcal B}}\langle \xi^*,\eta^* \rangle 
=\langle \xi,\eta \rangle_{{\mathcal B}}$ 
and $\langle \xi^*,\eta^* \rangle_{{\mathcal  A}}
={}_{{\mathcal A}}\langle \xi,\eta \rangle$.  
Then ${\mathcal E}^*$ is a ${\mathcal B}$-${\mathcal A}$ imprimitivity bimodule.  
The $C^{*}$-algebras ${\mathcal A}$ and ${\mathcal B}$ are called $Morita\  equivalent$ 
if there exists an ${\mathcal A}$-${\mathcal B}$ imprimitivity bimodule.  
For ${\mathcal A}$-${\mathcal B}$ imprimitivity bimodules ${\mathcal E}_{1}$ and ${\mathcal E}_{2}$, 
 ${\mathcal E}_{1}$ and ${\mathcal E}_{2}$ are called isomorphic  
if there exists a linear bijective map $\Phi$ from ${\mathcal E}_{1}$ onto ${\mathcal E}_{2}$ 
with the properties such that $\Phi(a\xi b)=a\Phi(\xi)b$, 
${}_{{\mathcal A}}\langle \Phi(\xi),\ \Phi(\eta)\rangle 
={}_{{\mathcal A}}\langle \xi,\eta\rangle$ 
and $\langle \Phi(\xi),\ \Phi(\eta)\rangle_{{\mathcal B}}
=\langle \xi, \eta\rangle_{{\mathcal B}}$ 
for any $a$ in ${\mathcal A}$, 
for any $b$ in ${\mathcal B}$ 
and for any $\xi, \eta$ in ${\mathcal E}_{1}$, 
where ${}_{{\mathcal A}}\langle \cdot ,\cdot \rangle$ and $\langle \cdot , \cdot \rangle_{{\mathcal B}}$ are left and right inner products respectively.
We denote by ${}_{{\mathcal A}}E_{{\mathcal B}}$ 
the set of isomorphic classes $[{\mathcal F}]$ of the ${\mathcal A}$-${\mathcal B}$ imprimitivity bimodule ${\mathcal F}$.\\ 
\ \ We recall some notations on Picard groups of $C^*$-algebras 
introduced by Brown, Green and Rieffel 
in \cite{BGR}.  
The set ${}_{{\mathcal A}}E_{{\mathcal A}}$ forms a group under the product by inner tensor product $\otimes $.  
The group, denoted by $\mathrm{Pic}(\c{A})$, are called $Picard\ group$ of $\c{A}$.  
The identity element of the group $\mathrm{Pic}(\c{A})$ is $[\c{A}]$, where $\c{A}$ is regarded as an ${\mathcal A}$-${\mathcal A}$ imprimitivity bimodule which have obvious left and right actions and the inner products ${}_\c{A}\langle a,b\rangle=ab^{*}$ and $\langle a,b \rangle_{\c{A}}=a^{*}b$.  
The $[\mathcal{E}^{*}]$ is the  
inverse element of $[\mathcal{E}]$ in the Picard group of $\c{A}$. 
Let $\alpha$ be an automorphism on $\c{A}$.  
We denote by $\mathcal{E}_{\alpha}$ the $\c{A}$-$\c{A}$-imprimitivity bimodule which is a set $\c{A}$ with obvious left actions , obvious left $\c{A}$-valued inner product and with following right actions and right $\c{A}$-valued inner product;  
$\xi\cdot a=\xi\alpha(a)$ for 
any $\xi\in\c{A}$, and $a\in \c{A}$,
$\langle\xi ,\eta\rangle_\c{A}=\alpha^{-1} (\xi^*\eta)$ for any $\xi ,\eta\in \c{A}$.  
For $\alpha, \beta\in\mathrm{Aut}(\c{A})$, 
$\mathcal{E}_\alpha$ is isomorphic to $\mathcal{E}_\beta$ if and only if 
there exists a unitary $u \in \c{A}$ such that 
$\alpha = ad \ u \circ \beta $. Moreover, ${\mathcal E}_\alpha \otimes 
{\mathcal E}_\beta$ is 
isomorphic to $\mathcal{E}_{\alpha\circ\beta}$.  We denote by $\rho_{\c{A}}$ the injective homomorphism 
from $\mathrm{Out}(\c{A})$ to $\mathrm{Pic}(\c{A})$ such that $\rho_{\c{A}}(\alpha)=[\c{E}_{\alpha}]$.  
Let $\mathcal{E}$ be an $\c{A}$-$\c{A}$-imprimitivity bimodule.  
If $\c{A}$ is unital, $\mathcal{E}$ has a finite basis $\{\xi_i\}_{i=1}^n$. 
Put $p=(\langle\xi_i,\xi_j\rangle_\c{A})_{ij} \in M_n(\c{A})$. 
Then $p$ is a full projection and $\mathcal{E}$ is isomorphic to 
$p\c{A}^n$ as $\c{A}$-$\c{A}$-imprimitivity bimodule 
with an isomorphism  of $\c{A}$ to  $pM_n(\c{A})p$ as $C^*$-algebra.  
Conversely, we suppose that $p$ is a full projection with an isomorphism $\alpha:\c{A}\rightarrow pM_{k}(\c{A})p$ 
and that the linear span of $\set{a^{*}pb \mid a,b\in {\mathcal A}^{n}}$ is dense in ${\mathcal A}$
(We call such projection $p$ a $self$-$similar\ full$ projection).  
Then $p\c{A}^{n}$ is an $\c{A}$-$\c{A}$-imprimitivity bimodule 
with the operations $a\cdot \xi=\alpha(a)\xi$, $\xi\cdot a=\xi a$.  

\section{Definition of Fundamental group}\label{sec:deffg}
 We define fundamental group for unital $C^*$-algebras 
with finite dimensional bounded trace space.  
The definitions given in the section 3 of \cite{NY} use a self-similariy of $C^*$-algebra.  
In the same way, we define a fundamental group.  
 
\begin{Def}\label{def:s.s}
Let ${\mathcal A}$ be a unital $C^*$-algebra.  
We denote by $(p,\Phi)$ the pair of self-similar full projection $p\in M_{k}(\c{A})$ and isomorphism $\Phi$ from $\c{A}$ onto $pM_{k}(\c{A})p$.  
We call the pair a $self$-$similar\ pair$.  
We abbrevate it as s.s.p.    
\end{Def}

We denote by $\r{Tr}_{k}$ the unnormalized trace on $M_{k}(\m{C})$.  
Let $\c{A}$ be a unital $C^{*}$-algebra, let $\varphi \in \r{ T}(\c{A})$ 
and let $(p,\Phi)$ be a self-similar pair of $\c{A}$.  
Then $(\r{Tr}_{k}\otimes \varphi) \circ \Phi$ is in $\r{ T}(\c{A})$.  
Therefore we can define a bounded linear map $T_{(p,\Phi)}$ on $\r{ T}(\c{A})$ by $T_{(p,\Phi)}(\varphi)=(\r{Tr}_{k}\otimes \varphi) \circ \Phi$.  

\begin{Def}\label{def:ontracesp}
Let $\c{A}$ be a unital $C^{*}$-algebra.  
We define a subset $\r{F}^{tr}({\mathcal A})$ of $\c{L}(\r{ T}(\c{A}))$ as follows;
\[\r{F}^{tr}({\mathcal A}):=\{ T_{(p,\Phi)}\in \c{L}(\r{ T}(\c{A})): (p,\Phi ): \r{ s.s.p}\}\]
\end{Def}

 We denote by $\c{GL}(\r{ T}(\c{A}))$ the set of invetible elements in $\c{L}(\r{ T}(\c{A}))$.  
It forms a group.  
We will show $\r{F}^{tr}({\mathcal A})$ is a subgroup of $\c{GL}(\r{ T}(\c{A}))$ by using a Picard group.  
The following construction generalizes that of the proposition 2.1 of \cite{NY}.  
\begin{pro}\label{pro:def}
Let ${\mathcal A}$ be a unital $C^*$-algebra.  
We define the map $R_{\mathcal A}:\r{ Pic}({\mathcal A})\rightarrow {\mathcal L} (\r{ T}({\mathcal A}))$ 
by $\left(R_{\mathcal A}([{\mathcal E}])(\varphi)\right)(a)=\sum_{i=1}^{k}\varphi(\langle \xi_{i}, a\xi_{i} \rangle_{})$, 
where $\{\xi_{i}\}_{i=1}^{k}$ is a finite basis of ${\mathcal E}$ as a right Hilbert ${\mathcal A}$-module.  
Then $R_{\mathcal A}([{\mathcal E}])$ does not depend on the choice of basis and $R_{\mathcal A}$ is well-defind.  Moreover $R_{\mathcal A}$ is multiplicative map and $R_{\mathcal A}([A])=id_{\r{ T}({\mathcal A})}$.  
\end{pro}
\begin{proof}
Let $\varphi$ be a trace on ${\mathcal A}$, 
$a$ be an element of ${\mathcal A}$, 
${\mathcal E}$ be an ${\mathcal A}-{\mathcal A}$ imprimitivity bimodule and let $\{\xi_{i}\}_{i=1}^{k}$ 
and $\{\eta_{j}\}_{j=1}^{l}$ be finite basis of ${\mathcal E}$.  
Then 
\begin{eqnarray}
\sum_{i=1}^{k}\varphi(\langle\xi_{i}, a\xi_{i} \rangle_{{\mathcal A}})&=&\sum_{i=1}^{k}\varphi(\langle\xi_{i}, \sum_{j=1}^{l}\eta_{j}\langle\eta_{j}, a\xi_{i} \rangle_{{\mathcal A}} \rangle_{{\mathcal A}}) 
= \sum_{i, j=1}^{k, l} \varphi(\langle \xi_{i}, \eta_{j} \rangle_{{\mathcal A}}\langle \eta_{j}, a\xi_{i} \rangle_{{\mathcal A}}) \nonumber \\
&=& \sum_{i, j=1}^{k, l} \varphi(\langle \eta_{j}, a\xi_{i} \rangle_{A}\langle \xi_{i}, \eta_{j} \rangle_{{\mathcal A}}) 
= \sum_{j=1}^{l}\varphi(\langle\eta_{j}, a\eta_{j} \rangle_{{\mathcal A}})\nonumber 
\end{eqnarray}
Therefore $R_{\mathcal A}([{\mathcal E}])$ is independent on the choice of basis.\\  
\ \ Let ${\mathcal E}_{1}$ and ${\mathcal E}_{2}$ be ${\mathcal A}$-${\mathcal A}$ imprimitivity bimodules 
with bases $\{\xi_{i}\}_{i=1}^{k}$ and $\{\zeta_{j}\}_{j=1}^{l}$.  
We suppose that there exists an isomorphism $\Phi$ of ${\mathcal E}_{1}$ onto ${\mathcal E}_{2}$.  
Then $\{\Phi(\xi_{i})\}_{i=1}^{k}$ is also a basis of ${\mathcal E}_{2}$.  
Then 
\begin{eqnarray}
\sum_{i=1}^{k}\varphi(\langle\xi_{i}, a\xi_{i} \rangle_{{\mathcal A}})=\sum_{i=1}^{k}\varphi(\langle\Phi(\xi_{i}), a\Phi(\xi_{i}) \rangle_{{\mathcal A}})
= \sum_{j=1}^{l}\varphi(\langle\zeta_{i}, a\zeta_{i} \rangle_{{\mathcal A}}) \nonumber 
\end{eqnarray}
Therefore $R_{\mathcal A}$ is well-defined.  \\
\ \ We shall show that $R_{\mathcal A}$ is multiplicative.   
Let ${\mathcal E}_{1}$ and ${\mathcal E}_{2}$ be ${\mathcal A}-{\mathcal A}$ imprimitivity  bimodules with bases $\{\xi_{i}\}_{i=1}^{k}$ and $\{\eta_{j}\}_{j=1}^{l}$.  
Then $\{\xi_{i}\otimes\eta_{j}\}_{i, j=1}^{k, l}$ is a basis of ${\mathcal E}_{1}\otimes{\mathcal E}_{2}$ and
\begin{eqnarray}
\left( R_{\mathcal A}([{\mathcal E}_{1}\otimes{\mathcal E}_{2}])\left(\varphi \right) \right) (a)=\sum_{i, j=1}^{k, l}\varphi(\langle \xi_{i}\otimes\eta_{j}, a\xi_{i}\otimes\eta_{j} \rangle)
=\sum_{i, j=1}^{k, l}\varphi(\langle \eta_{j}, \langle\xi_{i}, a\xi_{i} \rangle_{\mathcal A} \eta_{j} \rangle_{{\mathcal A}}) \nonumber
\end{eqnarray}
On the other hand
\begin{eqnarray}
\left(R_{\mathcal A}([{\mathcal E}_{1}])R_{\mathcal A}([{\mathcal E}_{2}])(\varphi)\right)(a) &=& \sum_{i=1}^{k}\left(R_{\mathcal A}([{\mathcal E}_{2}])\left(\varphi \right)\right)(\langle \xi_{i}, a\xi_{i} \rangle_{{\mathcal A}}) \nonumber \\
&=&\sum_{i, j=1}^{k, l}\varphi(\langle \eta_{j}, \langle\xi_{i}, a\xi_{i} \rangle_{{\mathcal A}} \eta_{j} \rangle_{{\mathcal A}}) \nonumber
\end{eqnarray}
Therefore $R_{\mathcal A}$ is multiplicative. 
\end{proof}

\begin{pro}\label{pro:extfundamental group}
Let ${\mathcal A}$ be a unital $C^*$-algebra.  
Then $\r{ F}^{tr}({\mathcal A})=R_{\mathcal  A}(\r{ Pic}({\mathcal A}))$.  
\end{pro}
\begin{proof}
Let ${\mathcal E}$ be an ${\mathcal A}$-${\mathcal A}$ imprimitivity bimodule 
and let $\{\xi_{i}\}^{k}_{i=1}$ be a basis of ${\mathcal E}$.  
Put $p=(\langle \xi_{i},\xi_{j}\rangle)_{ij}$.  
Then ${\mathcal E}$ is isomorphic to $p{\mathcal A}^{k}$ as an ${\mathcal A}$-${\mathcal A}$ imprimitivity bimodule and there exists an ${}^{*}$-isomomorphism $\alpha:{\mathcal A}\rightarrow pM_{k}({\mathcal A})p$.  
Then
\begin{eqnarray}
(R_{{\mathcal A}}([{\mathcal E}])(\varphi ))(a)&=&\sum^{k}_{i=1}\varphi (\langle \xi_{i},a\xi_{i} \rangle )
= \sum^{k}_{i=1}\varphi (\langle pe_{i}, \alpha (a)pe_{i} \rangle )\nonumber\\
&=& \sum^{k}_{i=1}\varphi ( e^{*}_{i}\alpha (a)e_{i})
= (\r{Tr}_{k}\otimes \varphi)\circ (\alpha)(a)\nonumber
\end{eqnarray}
Therefore $\r{ F}^{tr}({\mathcal A})\supset R_{{\mathcal A}}(\r{ Pic}({\mathcal A}))$.  
Conversely, we suppose that $p$ be a projection with an ${}^{*}$-isomorphism $\alpha:{\mathcal A}\rightarrow pM_{k}({\mathcal A})p$ and that the linear span of $\{a^{*}pb \mid a,b\in {\mathcal A}^{k}\}$ is dense in ${\mathcal A}$.    
Then, $p{\mathcal A}^{k}$ is an ${\mathcal A}$-${\mathcal A}$ imprimitivity bimodule 
with a basis $\{pe_{i}\}^{k}_{i=1}$.  
Then 
\begin{eqnarray}
(\r{Tr}_{k}\otimes \varphi)\circ (\alpha)(a)&=& \sum^{k}_{i=1}\varphi (\langle e^{*}_{i}, \alpha (a)e_{i} \rangle )
= \sum^{k}_{i=1}\varphi (\langle pe_{i}, \alpha (a)pe_{i} \rangle )\nonumber\\
&=&(R_{{\mathcal A}}([p{\mathcal A}^{k}])(\varphi ))(a)\nonumber
\end{eqnarray}
Therefore $\r{F}^{tr}({\mathcal A})\subset$ $R_{{\mathcal A}}(\r{ Pic}({\mathcal A}))$.  
Hence $\r{F}^{tr}({\mathcal A})$ $=$ $R_{{\mathcal A}}(\r{ Pic}({\mathcal A}))$.  
\end{proof}

\begin{pro}
Let $\c{A}$ be a unital $C^{*}$-algebra.  
Then $\r{ F}^{tr}(\c{A})$ is a subgroup of $\c{GL}(\r{ T}(\c{A}))$.  
\end{pro}

\begin{proof}
By \ref{pro:def}, $R_{\c{A}}(\r{ Pic}(\c{A}))$ is a subgroup of $\c{GL}(\r{ T}(\c{A}))$.  
Then, by \ref{pro:extfundamental group}, $\r{ F}^{tr}(\c{A})$ is also.  
\end{proof}

If $\r{ T}(\c{A})$ is finite dimensional, then $\r{ F}^{tr}(\c{A})$ has a matrix representation.   
In the case $\c{A}$ is unital, $T(\c{A})$ has a standard basis $\partial_{e}\r{T}(\c{A})^{+}_{1}$.  
We define the (determinant) fundamental group $\r{F}(\c{A})$ $(\r{F}_{\r{det}}(\c{A}))$ with this basis.   

\begin{Def}\label{def:fund}
Let $\c{A}$ be a unital $C^*$-algebra with finite dimensional bounded trace space.  
We define the fundamental group $\r{F}(\c{A})$ by the matrix representation of $\r{F}^{tr}(\c{A})$ with an indexed set of $\partial_{e}\r{T}(\c{A})^{+}_{1}$.  
The fundamental group is defined up to permutation of the index.  
Moreover, we define the determinant fundamental group $\r{F}_{\r{det}}(\c{A})$ by $|\r{det}|(\r{F}(\c{A}))$ 
where $|\r{det}|(X)=\set{|\r{det}(A)| \mid A\in X}$ for some subset $X$ of $M_{n}(\m{C})$.  
\end{Def} 

This fundamental group measures the elasticized degree of the self-similrity of a $C^{*}$-algebra by its trace space.  
By using the self-similarity, we shall show the following proposition.  

\begin{pro}\label{pro:tensor}
Let $\c{A}$ and $\c{B}$ be $C^{*}$-algebras with finite dimensional bounded trace space 
We suppose $\sharp(\partial_{e}T(\c{A})^{+}_{1})=n$ and that $\sharp(\partial_{e}T(\c{B})^{+}_{1})=m$.  
Say $\set{\varphi_{i}}^{n}_{i=1}=\partial_{e}T(\c{A})^{+}_{1}$ and $\set{\psi_{j}}^{m}_{j=1}=\partial_{e}\r{T}(\c{B})^{+}_{1}$.  \\
Then $\set{\varphi_{1}\otimes \psi_{1},\cdot \cdot \cdot ,\varphi_{n}\otimes \psi_{1},\varphi_{2}\otimes \psi_{1},\cdot \cdot \cdot, \varphi_{n}\otimes \psi_{m}}=\partial_{e}\r{ T}(\c{A}\otimes_{\r{min}}\c{B})^{+}_{1}$.  
Moreover \\ 
$\r{ F}(\c{A}\otimes _{\r{min}}\c{B})\supset \r{ F}(\c{A})\otimes \r{ F}(\c{B})$ 
and $\r{ F}_{\r{det}}(\c{A}\otimes _{\r{min}}\c{B})\supset |\r{det}|(\r{ F}(\c{A})\otimes \r{ F}(\c{B}))$, 
where $A\otimes B=\left[ 
 \begin{array}{ccc}
 Ab_{1,1} & \cdots & Ab_{1,m} \\
 \vdots & \ddots & \vdots \\
 Ab_{m,1} & \cdots & Ab_{m,m} \\
 \end{array} 
 \right]
$ for $A\in M_{n}(\m{C})$ and $(b_{ij})=B\in M_{n}(\m{C})$.  
\end{pro}
\begin{proof}  
Let $A\otimes B$ be an element of $\r{ F}(\c{A})\otimes \r{ F}(\c{B})$.  
Then there exist a projection $p_{0}$ in $M_{k_{1}}(\c{A})$ and a projection $q_{0}$ in $M_{k_{2}}(\c{B})$ 
such that $p_{0}M_{k_{1}}(\c{A})p_{0}\cong \c{A}$ and that $q_{0}M_{k_{2}}(\c{B})q_{0}\cong \c{B}$.  
Put $k=\r{ max}\set{k_{1},k_{2}}$, $p=\r{diag}(p_{0},0_{k-k_{0}})$ and $q=\r{diag}(q_{1},0_{k-k_{1}})$.  
Then $p\in M_{k}(\c{A})$, $q\in M_{k}(\c{B})$, $pM_{k}(\c{A})p\cong \c{A}$ and $qM_{k}(\c{B})q\cong \c{B}$.  
Therefore $(p\otimes q)(M_{k}(\c{A}\otimes_{\r{\min}} \c{B}))(p\otimes q)\cong \c{A}\otimes_{\r{min}} \c{B}$.  
\end{proof}
\begin{ex}\label{ex:tensor}
Put $\c{A}=M_{2^{\infty}}$ and $\c{B}=\m{C}\oplus \m{C}$.  \\
Then $\r{ F}(\c{A})=\set{2^{n}:n\in \m{Z}}$, $\r{ F}(\c{B})=\Set{
\left[ 
 \begin{array}{cc}
 1 & 0 \\
 0 & 1 \\
 \end{array} 
 \right],\ \left[ 
 \begin{array}{cc}
 0 & 1 \\
 1 & 0 \\
 \end{array} 
 \right]
}$, \\
$\r{ F}(\c{A}\otimes _{\r{min}}\c{B})=\Set{\left[ 
 \begin{array}{cc}
 2^{n} & 0 \\
 0 & 2^{m} \\
 \end{array} 
 \right],\ \left[ 
 \begin{array}{cc}
 0 & 2^{n} \\
 2^{m} & 0 \\
 \end{array} 
 \right]:n,m\in \m{Z}
}$
 , \\
$\r{ F}(\c{A})\otimes \r{ F}(\c{B})=\Set{
\left[ 
 \begin{array}{cc}
 2^{n} & 0 \\
 0 & 2^{n} \\
 \end{array} 
 \right],\ \left[ 
 \begin{array}{cc}
 0 & 2^{n} \\
 2^{n} & 0 \\
 \end{array} 
 \right]:n\in \m{Z}
}$, \\
$\r{F}_{\r{det}}(\c{A}\otimes _{\r{min}}\c{B})=\set{2^n :n\in \m{Z}}$ and 
$|\r{det}|(\r{ F}(\c{A})\otimes \r{ F}(\c{B}))=\set{4^{n}: n\in \m{Z}}$.  \\

Therefore $\r{F}(\c{A}\otimes _{\r{min}}\c{B})\supsetneq \r{ F}(\c{A})\otimes \r{ F}(\c{B})$ and $\r{ F}_{\r{det}}(\c{A}\otimes _{\r{min}}\c{B})\supsetneq |\r{det}|(\r{ F}(\c{A})\otimes \r{ F}(\c{B}))$.  
\end{ex}

 We define the fundamental group and the determinant fundamental group.  
Our definitions agree with 3.2 of \cite{NY}  
if the dimension of the bounded trace space is one.  
Let $\c{A}$ and $\c{B}$ be unital $C^{*}$-algebras with finite dimensional trace space.  
If $\c{A}$ and $\c{B}$ are Morita equivalent, then $\r{F}_{\r{det}}(\c{A})=\r{F}_{\r{det}}(\c{B})$.  
But $\r{ F}(\c{A})$ depends on the permitation of $\partial_{e}\r{T}(\c{A})^{+}_{1}$ if $\r{ dim}\r{ T}(\c{A})>1$.  
Moreover the fundamental groups $\r{ F}(\c{A})$ and $\r{ F}(\c{B})$ of two $C^*$-algebras $\c{A},\c{B}$  
which are Morita equivalent might be different in our definition(see \ref{ex:matrix}).  
In this section, we introduce the concept of being isomorphic and being weightedly isomorphic on the fundamental groups $\r{ F}(\c{A})$
considering the fact.  
We define the canonical unitary representation $U$ from a symmetric group $S_{n}$ into $M_{n}(\mathbb{C})$ 
by $U(\sigma)_{ij}=1$ if $j=\sigma(i)$ and $U(\sigma)_{ij}=0$ if $j\neq \sigma(i)$.  

\begin{Def}\label{def:isomorphism}
Let ${\mathcal A}$ and ${\mathcal B}$ be unital $C^*$-algebras \\
satisfying $\r{ dim}\r{ T}(\c{A})=\r{ dim}\r{ T}(\c{B})=n$.  
We call that $\r{  F}({\mathcal A})$ is isomorphic to $\r{ F}({\mathcal B})$ 
if there exists a permutation $\sigma$ in $S_{n}$ 
such that $\r{ F}({\mathcal B})=(U(\sigma))^{-1}\r{ F}({\mathcal A})(U(\sigma))$.
We call that $\r{ F}({\mathcal A})$ is weightedly isomorphic to $\r{ F}({\mathcal B})$ 
if there exists an invertible positive diagonal matrix $D$ in $M_{n}(\mathbb{C})$ 
and a permutation $\sigma$ in $S_{n}$ 
such that $\r{ F}({\mathcal B})=(DU(\sigma))^{-1}\r{ F}({\mathcal A})(DU(\sigma))$.    
\end{Def}
 We first consider fundamental groups of $C^*$-algebras which are isomorphic.  
\begin{pro}\label{pro:Ciso}
If two $C^*$-algebras, 
which have bounded trace spaces, 
are isomorhic, 
then their fundamental groups are same up to the permutation of basis.   
\end{pro}
\begin{proof}
Let ${\mathcal A}$ and ${\mathcal B}$ be unital $C^*$-algebras which have bounded trace spaces.  
We suppose $\sharp(\partial_{e}\r{T}(\c{A})^{+}_{1})=n$.  
Say $\{\varphi_{i}\}^{n}_{i=1}=\partial_{e}\r{T}(\c{A})^{+}_{1}$.     
If ${\mathcal A}$ and ${\mathcal B}$ are isomorphic, 
there exists an isomorphism $\alpha:{\mathcal A}\rightarrow{\mathcal B}$.  
Then $\{\varphi_{i}\circ \alpha \}^{n}_{i=1}=\partial_{e}\r{T}(\c{B})^{+}_{1}$.  
\end{proof}
 We next show that if two unital $C^*$-algebras $\c{A}$ and $\c{B}$ are Morita equivalent, 
then their fundamental groups $\r{ F}(\c{A})$ and $\r{ F}(\c{B})$ are weightedly isomorphic and $\r{ F}_{\r{det}}(\c{A})=\r{ F}_{\r{det}}(\c{B})$.   
Let ${\mathcal A}$ and ${\mathcal B}$ be unital $C^*$-algebras which have bounded trace spaces.  
If ${\mathcal A}$ and ${\mathcal B}$ are Morita equivalent, 
then there exists an ${\mathcal A}$-${\mathcal B}$ imprimitivity bimodule ${\mathcal F}$.  
We define the linear map 
$R_{{\mathcal A}{\mathcal B}}:{}_{\mathcal A}E_{\mathcal B}\rightarrow {\mathcal L}(\r{ T}({\mathcal B}),\r{ T}({\mathcal A}))$ 
as $(R_{{\mathcal A}{\mathcal B}}([{\mathcal F}])(\varphi))(a)
=\sum_{i=1}^{k}\varphi(\langle \xi_{i}, a\xi_{i} \rangle_{{\mathcal B}})$, 
where $\{\xi_{i}\}_{i=1}^{k}$ be the basis of ${\mathcal F}$ as a right Hilbert ${\mathcal B}$-module.    
\begin{lem}\label{lem:hom}
Let ${\mathcal A}$, ${\mathcal B}$ and ${\mathcal C}$ be unital $C^*$-algebras which have bounded trace spaces.  
We suppose ${\mathcal A}$, ${\mathcal B}$ and ${\mathcal C}$ are Morita equivalent.  
Let ${\mathcal E}$ be an imprimitivity ${\mathcal A}$-${\mathcal B}$ bimodule 
and let ${\mathcal F}$ be an imprimitivity ${\mathcal B}$-${\mathcal C}$ bimodule.  
Then $R_{{\mathcal A}{\mathcal B}}([{\mathcal E}])R_{{\mathcal B}{\mathcal C}}([{\mathcal F}])
=R_{{\mathcal A}{\mathcal C}}([{\mathcal E}\otimes{\mathcal F}])$.  
In particular,  
$R_{{\mathcal A}{\mathcal B}}([{\mathcal F}])R_{{\mathcal B}{\mathcal A}}([{\mathcal F}^*])=id_{\r{ T}({\mathcal A})}$, $R_{{\mathcal B}{\mathcal A}}([{\mathcal F}^*])R_{{\mathcal A}{\mathcal B}}([{\mathcal F}])=id_{\r{ T}({\mathcal B})}$ and $\r{ dim}\r{ T}(\c{A})=\r{ dim}\r{ T}(\c{B})$.   
\end{lem}

\begin{proof}
As in the same proof \ref{pro:def}, 
$R_{{\mathcal A}{\mathcal B}}([{\mathcal F}])$ is independent of the choice of the basis of ${\mathcal F}$, 
$R_{{\mathcal A}{\mathcal B}}$ is well-defined, 
and next proposition follows.  
\end{proof}

Especially,  $R_{{\mathcal A}{\mathcal A}}([{\mathcal E}])=R_{\mathcal A}([{\mathcal E}])$ 
where ${\mathcal E}$ is a ${\mathcal A}$-${\mathcal A}$ imprimitivity bimodule.  

\begin{pro}\label{pro:morita}
Let ${\mathcal A}$ and ${\mathcal B}$ be unital $C^*$-algebras with finite dimensional bounded trace space.  
If ${\mathcal A}$ and ${\mathcal B}$ are Morita equivalent,  
then $\r{ F}({\mathcal A})$ is weightedly isomorphic to $\r{ F}({\mathcal B})$ and $\r{ F}_{\r{det}}(\c{A})=\r{ F}_{\r{det}}(\c{B})$.  
\end{pro}
\begin{proof}
Let ${\mathcal F}$ be an ${\mathcal A}$-${\mathcal B}$ imprimitivity bimodule.  
Then ${\mathcal F}$ induces an isomorphism ${\Psi}$ of $\r{ Pic}({\mathcal A})$ to $\r{ Pic}({\mathcal B})$ 
such that ${\Psi}([{\mathcal E}])=[{\mathcal F}^*\otimes{\mathcal E}\otimes{\mathcal F}]$ 
for $[{\mathcal E}]\in \r{ Pic}({\mathcal A})$.  By \ref{lem:hom}, for $[{\mathcal E}] \in \r{ Pic}({\mathcal A})$ 
\begin{eqnarray} 
R_{\mathcal B}([{\mathcal F}^*\otimes{\mathcal E}\otimes{\mathcal F}])&=&R_{{\mathcal B}{\mathcal B}}([{\mathcal F}^*\otimes{\mathcal E}\otimes{\mathcal F}]) \nonumber \\
&=&R_{{\mathcal B}{\mathcal A}}([{\mathcal F}^*])R_{{\mathcal A}{\mathcal A}}([{\mathcal E}])R_{{\mathcal A}{\mathcal B}}([{\mathcal F}]) \nonumber \\
&=&R_{{\mathcal B}{\mathcal A}}([{\mathcal F}^*])R_{\mathcal A}([{\mathcal E}])R_{{\mathcal A}{\mathcal B}}([{\mathcal F}]) \nonumber
\end{eqnarray}
Put $\r{ dim}(\r{ T}({\mathcal A}))=\r{ dim}(\r{ T}({\mathcal B}))=n$.  
We shall consider the representation of the previous formula into $M_{n}(\mathbb{C})$ 
with the basis $\partial_{e}\r{T}(\c{A})^{+}_{1}$ of $\r{ T}({\mathcal A})$ 
and with the basis $\partial_{e}\r{T}(\c{B})^{+}_{1}$ of $\r{ T}({\mathcal B})$.  
Then the representation matrices of $R_{\mathcal A}([{\mathcal E}])$ 
and $R_{\mathcal B}([\Psi({\mathcal E})])$ are $S_{\mathcal A}(R_{\mathcal A}([{\mathcal E}]))$ 
and $S_{\mathcal A}(R_{\mathcal B}([\Psi({\mathcal E})]))$ respectively.  
By definition of $R_{{\mathcal A}{\mathcal B}}({\mathcal F})$ 
and $R_{{\mathcal B}{\mathcal A}}({\mathcal F}^*)$, 
they preserve positiveness of traces, 
so the representation matrix elements of them are positive.  
Moreover each of them are inverse element of the other by \ref{lem:hom}.  
Therefore the representation matrix of $R_{{\mathcal A}{\mathcal B}}({\mathcal F})$ has the form of \ref{lem:positivematrix}.  
Hence there exists a positive invetible diagonal matrix $D$ 
and exists $\sigma$ in $S_{n}$ 
such that $S_{\mathcal B}(R_{\mathcal B}([\Psi({\mathcal E})]))
=(DU(\sigma))^{-1}S_{\mathcal A}(R_{\mathcal A}({\mathcal F}))(DU(\sigma))$.  
Furthermore, since $\r{det}(P^{-1}AP)=\r{det}(A)$ for $P\in GL_{n}(\m{C})$ 
and for any $A\in M_{n}(\m{C})$, $\r{ F}_{\r{det}}(\c{A})=\r{ F}_{\r{det}}(\c{B})$.        
\end{proof}

\section{Forms of fundamental groups and some example}

 We shall show that $\r{F}({\mathcal A})$ is restricted by K-theoretical obstruction and positivity.  
This computation was motivated by proposition 3.7 of \cite{NY}.  
We denote by $\varphi^*$ the map from $K_{0}({\mathcal A})$ into $\mathbb{R}$ 
induced by a bounded trace $\varphi$ on ${\mathcal A}$.  
Furthermore we denote by $H({\mathcal A})$ the set of isomorphic classes $[\chi]$ 
of right Hilbert ${\mathcal A}$-modules $\chi$ with finite basis $\set{\xi_{i}}^{k}_{i=1}$.  
We define a pairing $\braket{\cdot ,\cdot }:H({\mathcal A})\times \r{ T}(\c{A})\rightarrow \m{C}$ 
by $\braket{[\chi],\varphi}=\sum^{k}_{i=1}\varphi(\braket{\xi_{i},\xi_{i}}_{\c{A}})$.  

\begin{pro}
Let ${\mathcal A}$ be a unital $C^*$-algebra with finite dimensional bounded trace space.  
Then 
\begin{eqnarray}
\braket{[\chi\otimes \c{E}],\varphi}=\braket{[\chi], R_{\c{A}}([\c{E}])(\varphi)}\nonumber
\end{eqnarray} 
for any right Hilbert ${\mathcal A}$-modules $\chi$ with finite basis, 
for any ${\mathcal A}$-${\mathcal A}$ imprimitivity bimodule $\c{E}$ and 
for any $\varphi\in \r{ T}(\c{A})$.  
\end{pro}
\begin{proof}
Let $\set{\xi_{i}}^{k}_{i=1}$ be a finite basis of $\chi$ 
and let $\set{\eta_{j}}^{l}_{j=1} $ be a finite basis of $\c{E}$.  
Then $\braket{[\chi\otimes \c{E}],\tau}=\braket{[\chi], R_{\c{A}}([\c{E}])(\tau)}=\sum_{i, j=1}^{k, l}\varphi(\langle \eta_{j}, \langle\xi_{i}, \xi_{i} \rangle_{{\mathcal A}} \eta_{j} \rangle_{{\mathcal A}})$.  
\end{proof}
The following computations are based on the pairing.  
  
\begin{lem}\label{lem:Krelationmap}
Let ${\mathcal A}$ be a unital $C^*$-algebra with finite dimensional bounded trace space.  
We define the map $\hat{R}_{\mathcal A}:H({\mathcal A})\rightarrow \r{ T}({\mathcal A})^*$ 
by $\hat{R}_{\mathcal A}([\chi])(\varphi)=\sum_{i=1}^{k}\varphi(\langle \xi_{i}, \xi_{i}\rangle)$, 
where $\{\xi_{i}\}_{i=1}^{k}$ is a finite basis of ${\mathcal E}$ as a right Hilbert ${\mathcal A}$-module.  
If $[{\mathcal E}]$ is an element of $\r{ Pic}({\mathcal A})$, 
then $\hat{R}_{\mathcal A}([\chi \otimes_{{\mathcal A}} {\mathcal E}])=\hat{R}_{\mathcal A}([\chi])R_{\mathcal A}([{\mathcal E}])$.  
\end{lem}
\begin{proof}
As in the same proof \ref{pro:def}, 
$\hat{R}_{\mathcal A}([\chi])$ does not depend on the choice of basis, 
$\hat{R}_{\mathcal A}$ is well-defined, 
and $\hat{R}_{\mathcal A}([\chi \otimes_{{\mathcal A}} {\mathcal E}])=\hat{R}_{\mathcal A}([\chi])R_{\mathcal A}([{\mathcal E}])$ for all ${\mathcal A}-{\mathcal A}$ imprimitivity bimodule ${\mathcal E}$.   
\end{proof}

We denote by $P_{k}(\c{A})$ the set of projections in $M_{k}(\m{C})$ 
and put $P_{\infty}(\c{A})=\cup^{\infty}_{k=1}P_{k}(\c{A})$.  
 
\begin{lem}\label{lem:HK}
Let ${\mathcal A}$ be a unital $C^*$-algebra 
and let $\varphi$ be a positive bounded trace on ${\mathcal A}$.  
Then
\begin{eqnarray}
\{\varphi^*([p]_{0}):p\in P_{\infty}({\mathcal A})\}&=&\{\varphi\otimes {\rm Tr}_{k} (p):p\in M_{k}({\mathcal A}), k\in \mathbb{N}\}\nonumber \\
&=&\{\hat{R}_{\mathcal A}([\chi])(\varphi):[\chi] \in H({\mathcal A})\}\nonumber
\end{eqnarray}
\end{lem}
\begin{proof}
Let $\chi$ be a right Hilbert ${\mathcal A}$-module with finite basis $\{\xi_{i}\}_{i=1}^{k}$.  
Put $p=(\langle\xi_{i}, 
\xi_{j}\rangle_{{\mathcal A}})_{ij}$ be a projection in $M_{k}({\mathcal A})$.  
Then $\varphi(p)=\hat{R}_{\mathcal A}([\chi])(\varphi)$.  
On the other hand, let $p$ be a projection in $M_{k}({\mathcal A})$.  
Then $p{\mathcal A}^{k}$ is a right Hilbert ${\mathcal A}$-module 
with finite basis $\{pe_{i}\}_{i=1}^{k}$ 
where $\set{e_{i}}^{k}_{i=1}$ is a canonical basis of ${\mathcal A}^{k}$.  
Therefore $\hat{R}_{\mathcal A}([p{\mathcal A}^{k}])(\varphi)=\varphi(p)$.  
\end{proof}
\begin{lem}\label{lem:Krelationmap2}
Let ${\mathcal A}$ be a unital $C^*$-algebra with finite dimensional bounded trace space.  
We suppose $\partial_{e}\r{T}(\c{A})^{+}_{1}=n$.  
Say $\{\varphi_{i}\}_{i=1}^{n}=\partial_{e}\r{T}(\c{A})^{+}_{1}$.  
Put $E=\oplus_{i=1}^{n}\varphi^*_{i}(K_{0}({\mathcal A}))$.  
Then we can consider $E$ 
as an additive subgroup of $M_{1,n}(\mathbb{C})$.  
Then $EB=E$ 
for all $B$ in $\r{F}({\mathcal A})$.  
In particular,  $E$ is a module over $\r{F}({\mathcal A})$.  
\end{lem}
\begin{proof}
Let $B$ be an element in $\r{F}({\mathcal A})$.  
Put $E_{0}=\oplus_{i=1}^{n}\{\hat{R}_{\mathcal A}([{\mathcal E}])(\varphi_{i}):[{\mathcal E}] \in H({\mathcal A})\}$
Considering the representation of the equation in \ref{lem:Krelationmap} 
with a basis of $\partial_{e}\r{T}(\c{A})^{+}_{1}$, 
$E_{0}B
\subset E_{0}$.  
Since $B$ is invertible, $E_{0}B=E_{0}$.  
By \ref{lem:HK} and by the standard picture of $K_{0}$, 
$E_{0}=E$.  
Hence $EB=E$.  
\end{proof}

We denote by $M_{n}(\m{R}^{+})$ the set of matrices all elements of which are  nonnegative.  

\begin{lem}\label{lem:positivematrix}
Let $A$ be an invertible element of $M_{n}(\mathbb{C})$.  
If $A$ and $A^{-1}$ are elements of $M_{n}(\mathbb{R}^{+})$, 
then there exist a permutation $\sigma$ in $S_{n}$ 
and a positive invertible diagonal matrix $D$,  
such that $A=DU(\sigma)$.  
\end{lem}
\begin{proof}
This proof is based on the calculation of matrix elements.  
It is sufficient to show that there exists only one positive element in each row and in each column.  
Let $A=\{a_{ij}\}, B=\{b_{ij}\}$ be an element of $M_{n}(\mathbb{R}^{+})$ such that $AB=BA=E_{n}$.  
We regard $i_{0}$ as fixed and we suppose $a_{i_{0}j}=0$ for all $j$.  
Then $\r{det}A=0$, which contradicts the invertibility of $A$.  
Therefore there exists $j_{0}$ such that $a_{i_{0}j_{0}}>0$.  
We shall show that if $j\neq j_{0}$, then $a_{i_{0}j}=0$.  
Since $AB=E_{n}$, if $j\neq i_{0}$, $\sum_{k=1}^{n}a_{i_{0}k}b_{kj}=0$.  
By the positivity of the matrix elements of $A$ and $B$, 
if $j\neq i_{0}$, then $b_{j_{0}j}=0$.  
Otherwise, we suppose $b_{j_{0}i_{0}}=0$, 
then $\r{det}B=0$, which contradicts the invertibility of $B$.  
Therefore $b_{j_{0}i_{0}}>0$.  
Since $BA=E_{n}$, if $j\neq j_{0}$, then $\sum_{k=1}^{n}b_{j_{0}k}a_{kj}=0$.  
Because $b_{j_{0}i_{0}}>0$, if $j\neq j_{0}$, then $a_{i_{0}j}=0$.  
Hence there exists only one positive element in each column.  
By transposing the matrices on the both side of  $AB=BA=E_{n}$, 
the rest of the proof runs as before.   
\end{proof}

By the above lemmas, the form of the elements in $\r{ F}(\c{A})$ are restricted.   

\begin{pro}\label{pro:Krelation}
Let $A$ be a unital $C^*$-algebra with finite dimensional bounded trace space.  
We suppose $\partial_{e}\r{T}(\c{A})^{+}_{1}=n$.  
Say $\set{\varphi_{i}}_{i=1}^{n}=\partial_{e}\r{T}(\c{A})^{+}_{1}$.  
For any $B$ in $\r{F}({\mathcal A})$, 
there exists an invertible diagonal matrix $D$ in $M_n(\mathbb{R^{+}})$ 
and $\sigma \in S_{n}$ 
such that $B=DU(\sigma)$ and 
$D_{ii}\varphi_{i}^*(K_{0}({\mathcal A}))=\varphi_{\sigma(i)}^*(K_{0}({\mathcal A}))$.  
\end{pro}
\begin{proof}
We first prove that for any element $B$ in $\r{F}({\mathcal A})$ 
there exist a $\sigma$ in $S_{n}$ and a positive diagonal matrix $D$, 
which belong to $M_{n}(\mathbb{R}^{+})$, 
such that $B=DU(\sigma)$.  
Let $\varphi$ be an element in $\r{ T}({\mathcal A})^{+}$, 
let ${\mathcal E}$ be an ${\mathcal A}-{\mathcal A}$ imprimitivity bimodule, 
and let $\{\xi_{i}\}_{i=1}^{n}$ be a basis of ${\mathcal E}$.  
Because $\varphi$ is positive, 
$R_{\mathcal A}([{\mathcal E}])(\varphi)$ is positive.  
Therefore considering the representation of $R_{\mathcal A}([{\mathcal E}])$ and $R_{\mathcal A}([{\mathcal E}^*])$ 
with the basis of $\r{ T}({\mathcal A})_{1}^{+}$, 
we see that all elements of the matrices are positive.  
By \ref{lem:positivematrix}, the first step is proved and we can put $B=DU(\sigma)$.  \\
\ \ We next prove $D_{ii}\varphi_{i}^*(K_{0}({\mathcal A}))=\varphi_{\sigma(i)}^*(K_{0}({\mathcal A}))$.  Put $E_{i}=\varphi_{i}^*(K_{0}({\mathcal A}))$.  
By \ref{lem:Krelationmap2} and by the previous result of $B=DU(\sigma)$, 
we can obtain $D_{ii}E_{i}
\subset E_{\sigma(i)}$.  
Considering the invertibility of B in $\r{F}({\mathcal A})$, 
$\cfrac{1}{D_{ii}} E_{\sigma(i)}
\subset E_{i}$.  
Hence $D_{ii} E_{i}^*(K_{0}({\mathcal A}))=E_{\sigma(i)}$.   
\end{proof}

\begin{rem}\label{rem:extreme ray}
Let ${\mathcal A}$ and ${\mathcal B}$ be a unital $C^*$-algebra.  
Suppose $\r{ dim}\r{ T}({\mathcal A})=\r{ dim}\r{ T}({\mathcal B})$.  
We say a positive bounded trace $\varphi$ is an extreme ray 
if for any fixed positive bounded trace $\psi$, 
if $\psi\leq \varphi$, 
then there exists a positive real number $\lambda $ 
such that $\lambda\varphi= \psi$. 
If both an invertible element $T$ of ${\mathcal L}(\r{ T}({\mathcal A}),\r{ T}({\mathcal B}))$ 
and the inverse $T^{-1}$ are positive, 
then $T\varphi$ is an extreme ray of $\r{ T}({\mathcal B})^{+}$ 
for any extreme ray $\varphi$ of $\r{ T}({\mathcal A})^{+}$.  
We can also prove the previous proposition by the fact.  
\end{rem}

\begin{cor}
Let ${\mathcal A}$ be a unital $C^*$-algebra with finite dimensional bounded trace space.  
If $\c{A}$ is separable, then $\r{ F}(\c{A})$ and $\r{F}_{\r{det}}(\c{A})$ are countable group.  
\end{cor}
\begin{proof}
We suppose $\partial_{e}\r{T}(\c{A})^{+}_{1}=n$.  
Say $\{\varphi_{i}\}_{i=1}^{n}=\partial_{e}\r{T}(\c{A})^{+}_{1}$.  
Then $\oplus^{n}_{i=1}\varphi^{*}_{i}(K_{0}(\c{A}))$ is a countable additive subgroup of $\m{R}^{n}$.  
Let $B\in \r{ F}(\c{A})$.  
There exists an invertible diagonal matrix $D$ in $M_n(\mathbb{R^{+}})$ 
and $\sigma \in S_{n}$ 
such that $B=DU(\sigma)$
Since $1\in \varphi^{*}_{i}(K_{0}(\c{A}))$ for any $i$, 
$D_{ii}\in \varphi^{*}_{\sigma(i)}(K_{0}(\c{A}))$ for any $k,l$.  
Therefore $\r{ F}(\c{A})$ and $F_{\r{det}}(\c{A})$ are countable.  
\end{proof}

 That proposition \ref{pro:Krelation} enables us to calculate $\r{F}({\mathcal A})$ easily.  
We shall show some examples.  

\begin{cor}\label{cor:dsum}
Let ${\mathcal A}$ be a unital $C^*$-algebra with a unique normalized bounded trace $\tau$, 
let ${\mathcal B}\equiv\oplus_{i=1}^{n}{\mathcal A}_{i}$ 
where ${\mathcal A}_{i}={\mathcal A}$.  
If $\r{F}({\mathcal A})=\tau^*(K_{0}({\mathcal A}))\cap(\tau^*(K_{0}({\mathcal A})))^{-1}\cap \m{R}^{+}$,  
then \[\r{F}({\mathcal B})=\{DU(\sigma):\sigma\in S_{n}, 
D\ \r{ is\ a\ diagonal\ matrix\ in}\ M_{n}(\mathbb{R^{+}})\ s.t. \ D_{ii}\in \r{F}({\mathcal A})\}\] 
\end{cor}
\begin{proof}
Let $\tau_{i}$ be the normalized trace of ${\mathcal A}_{i}$.  
We define $\alpha(\sigma):{\mathcal B}\rightarrow {\mathcal B}$ \\
by $\alpha(\sigma)(a_{1}, a_{2}, \cdot\cdot\cdot, a_{n})
=(a_{\sigma^{-1}(1)}, a_{\sigma^{-1}(2)}, \cdot\cdot\cdot, a_{\sigma^{-1}(n)})$ 
where $\sigma\in S_{n}$.  
Considering $[{\mathcal E}_{\alpha(\sigma)}]$ in $\r{ Pic}({\mathcal B})$, 
$\r{F}({\mathcal B})$ includes permitations of $M_{n}(\mathbb{R})$.  
Moreover $\{D:\r{ diagonal\ matrix\ in}\ M_{n}(\mathbb{R^{+}}):D_{ii}\in \r{F}({\mathcal A})\}$ 
is a subset(subgroup) of $\r{F}({\mathcal B})$.  
Hence (RHS)$\subset \r{F}({\mathcal B})$.  
As a result, by \ref{pro:Krelation}, 
$\r{F}({\mathcal B})\subset$(RHS).  
\end{proof}  
\begin{cor}\label{cor:0}
Let ${\mathcal A}$ be  unital $C^*$-algebras with finite dimensional bounded trace space.  
We suppose $\partial_{e}\r{T}(\c{A})^{+}_{1}=n$.  
Say $\{\varphi_{i}\}_{i=1}^{n}=\partial_{e}\r{T}(\c{A})^{+}_{1}$.  
If there exist $i_{0}$ and $i_{1}$ 
such that $d\varphi_{i_{0}}^*(K_{0}({\mathcal A}))
\neq\varphi_{i_{1}}^*(K_{0}({\mathcal A}))$ 
for any  $d\in\mathbb{R}^{+}$, 
then $B_{i_{0}i_{1}}=B_{i_{1}i_{0}}=0$ 
for all element $B$ in $\r{F}({\mathcal A})$.   
\end{cor}
\begin{proof}
We suppose $B_{i_{0}i_{1}}\neq0$.  
Then $B_{i_{0}i_{1}}\varphi_{i_{1}}^*(K_{0}({\mathcal A}))=\varphi_{i_{0}}^*(K_{0}({\mathcal A}))$.  
This contradicts our assumption.   
\end{proof}
\begin{ex}\label{ex:prime}
Put ${\mathcal A}=M_{2^{\infty}}\oplus M_{2^{\infty}}\oplus M_{3^{\infty}}$, \\
then 
$\r{F}({\mathcal A})=\{\left[ 
 \begin{array}{ccc}
2^{k} &0 & 0\\
0 & 2^{l}& 0\\
0 & 0& 3^{m}\\
 \end{array} 
 \right], 
 \left[ 
 \begin{array}{ccc}
0 &2^{k} & 0\\
2^{l} & 0& 0\\
0 & 0& 3^{m}\\
 \end{array} 
 \right]
: k,l,m \in \mathbb{Z}
\}$\\ 
and $\r{F}_{\r{det}}(\c{A})=\set{2^{n}3^m:n,m\in \m{Z}}$.  
\end{ex}
\begin{ex}\label{ex:simpleaf}
Let $P$ be the set of all prime numbers.  \\
Put $P=\set{p_{i}:i\in \m{N}\ p_{i}<p_{j}(i<j)}$.  
Let ${\mathcal A}_{n}=M_{\Pi^{n}_{k=1}(4p_{k}^{2})}(\m{C})\oplus M_{\Pi^{n}_{k=1}(4p_{k}^{2})}(\m{C})$.  
We define ${}^*$-homomorphisms $\psi_{n} :A_{n}\rightarrow A_{n+1}$\\
by $\psi_{n}((a,b))=(\r{diag}(a,a,\cdots,a,b,b),\r{diag}(b,b,\cdots,b,a,a))$.  
Let ${\mathcal A}$ be the inductive limit of the sequence
\[\begin{CD}
      {\mathcal A}_{0} @>\psi_{0}>> {\mathcal A}_{1} @>\psi_{1}>> {\mathcal A}_{2} @>>> \cdots  \end{CD}. \]
\end{ex} 
Then ${\mathcal A}$ is simple and $\r{ dim}\r{ T}({\mathcal A})=2$.  
Let $\tau^{(n)}_{1}$ and $\tau^{(n)}_{2}$ be the normalized traces on $\c{A}_{n}$ such that 
$\tau^{(n)}_{1}=\cfrac{1}{\Pi^{n}_{k=1}(4p_{k}^{2})}\r{Tr}_{\Pi^{n}_{k=1}(4p_{k}^{2})}\oplus 0$ and
$\tau^{(n)}_{2}=0\oplus \cfrac{1}{\Pi^{n}_{k=1}(4p_{k}^{2})}\r{Tr}_{\Pi^{n}_{k=1}(4p_{k}^{2})}$.  
Let $\varphi_{1},\varphi_{2}$ be different elements of $\partial_{e}\r{T}({\mathcal A})^{+}_{1}$ 
which satisfy 
\begin{eqnarray}
\varphi_{1}|_{\c{A}_{n}}=(\cfrac{1}{2}+ \cfrac{3\Pi^{n}_{k=1} p_{k}^{2}}{(\Pi^{n}_{k=1}(p_{k}^{2}-1))\pi^{2}})\tau^{(n)}_{1}\oplus (\cfrac{1}{2}- \cfrac{3\Pi^{n}_{k=1} p_{k}^{2}}{(\Pi^{n}_{k=1}(p_{k}^{2}-1))\pi^{2}})\tau^{(n)}_{2}\nonumber
\end{eqnarray}
and 
\begin{eqnarray}
\varphi_{2}|_{\c{A}_{n}}=(\cfrac{1}{2}- \cfrac{3\Pi^{n+1}_{k=1} p_{k}^{2}}{(\Pi^{n}_{k=1}(p_{k}^{2}-1))\pi^{2}})\tau^{(n)}_{1}\oplus (\cfrac{1}{2}+ \cfrac{3\Pi^{n}_{k=1} p_{k}^{2}}{(\Pi^{n}_{k=1}(p_{k}^{2}-1))\pi^{2}})\tau^{(n)}_{2}.\nonumber
\end{eqnarray}   
We compute the fundamental group of $\c{A}$ and we shall show \\
$\r{ F}(\c{A})=\Set{\left[ 
 \begin{array}{cc}
 2^{n}& 0 \\
 0& 2^{n} \\
 \end{array} 
 \right], 
 \left[ 
 \begin{array}{cc}
 0& 2^{n} \\
 2^{n}& 0 \\
 \end{array} 
 \right](n\in \m{Z})}$ 
and $F_{\r{det}}(\c{A})=\set{4^{n}:n\in \m{Z}}$.  
\begin{lem}\label{lem:examplee}
Let $\c{A}$ be the inductive limit of the above sequence.  
Let $\varphi_{1},\varphi_{2}$ as above.  
Then both $\varphi^*_{1}(K_{0}({\mathcal A}))$ and $\varphi^*_{2}(K_{0}({\mathcal A}))$ 
are the same additive group $E$.  
If there exists a positive real number $\lambda$ 
such that $\lambda E=E$, 
then $\lambda =2^{n}(n\in \m{Z})$.
\end{lem}
\begin{proof}
Then both $\varphi^*_{1}(K_{0}({\mathcal A}))$ and $\varphi^*_{2}(K_{0}({\mathcal A}))$ 
are the same additive group $E$ generated by 
$\cfrac{1}{2^{2n+1}\Pi^{n}_{k=1} p_{k}^{2}}\pm \cfrac{3}{4^{n}(\Pi^{n}_{k=1}(p_{k}^{2}-1))\pi^{2}}(n\in \mathbb{N})$. 
Since $1 \in  E$, 
there exist rational numbers $q_{1},q_{2},q_{3},q_{4}$ 
such that $\lambda=q_{1}+\cfrac{q_{2}}{\pi^{2}}$, 
$(q_{1}+\cfrac{q_{2}}{\pi^{2}})(q_{3}+\cfrac{q_{4}}{\pi^{2}})=1$.  
Since $\pi$ is a transcendental number, $q_{2}=q_{4}=0$.  
Then $\lambda$ is a rational number.  
Put $\lambda=2^{a}\cdot \cfrac{l}{m}$, where $a$ is an integer 
and $l,m$ are non-zero positive odd numbers satisfying $gcd(l,m)=1$.  
We will show $l=m=1$.  
Conversely, suppose $l\neq 1$.  
Then there exist a prime number $p_{n_{0}}$ and an integer $l_{1}$ such that $l=p_{n_{0}}l_{1}$.  
Since $gcd(l,m)=1$, $\cfrac{1}{2^{2n_{0}+1}\Pi^{n_{0}}_{k=1}p^{2}_{k}}+\cfrac{3}{4^{n_{0}}(\Pi^{n_{0}}_{k=1}(p_{k}^{2}-1))\pi^{2}} \not\in \lambda E$.  
This contradicts to the fact $\lambda E=E$.  
Suppose $m\neq 1$. Similarly, we can denote $m=p_{n_{1}}m_{1}$, 
where $p_{n_{1}}$ is a prime number and $m_{1}$ is an integer.   
Then $\cfrac{1}{m}(\cfrac{1}{2^{2n_{1}+1}\Pi^{n_{1}}_{k=1}p^{2}_{k}}+\cfrac{3}{4^{n_{1}}(\Pi^{n_{1}}_{k=1}(p_{k}^{2}-1))\pi^{2}}) \not\in E$.  
We next show $\lambda=2^{m}(m\in \m{Z})$, then $\lambda E=E$.  
It is sufficient to show the case $m>0$.  
Let $\alpha=\cfrac{1}{2^{2n+1}\Pi^{n}_{k=1} p_{k}^{2}}\pm \cfrac{3}{4^{n}(\Pi^{n}_{k=1}(p_{k}^{2}-1))\pi^{2}}$ be a generator of $E$.  
For an integer $n_{0}$ more than $\cfrac{m}{2}+n$, 
$\alpha=2^{m}\cdot (\cfrac{2^{2n_{0}-m-2n}\Pi^{n_{0}}_{k=n+1}p^{2}_{k}}{2^{2n_{0}+1}\Pi^{n_{0}}_{k=1} p_{k}^{2}}\pm \cfrac{3(2^{2n_{0}-m-2n}\Pi^{n_{0}}_{k=n+1}(p^{2}_{k}-1))}{4^{n_{0}}(\Pi^{n_{0}}_{k=1}(p_{k}^{2}-1))\pi^{2}})$
Therefore $\lambda E\supset E$.  
Obviously, $\lambda E\subset E$, then $\lambda E=E$.  
\end{proof}
\begin{pro}\label{pro:examplere}
Let $\c{A}$ be the inductive limit of the sequence in \ref{ex:simpleaf}.\\
Then $\r{ F}(\c{A})=\Set{\left[ 
 \begin{array}{cc}
 2^{n}& 0 \\
 0& 2^{n} \\
 \end{array} 
 \right], 
 \left[ 
 \begin{array}{cc}
 0& 2^{n} \\
 2^{n}& 0 \\
 \end{array} 
 \right](n\in \m{Z})}$ 
and $\r{F}_{\r{det}}(\c{A})=\set{4^{n}:n\in \m{Z}}$.  
\end{pro}
\begin{proof}
We first show $\left[ 
 \begin{array}{cc}
 0& 1 \\
 1& 0 \\
 \end{array} 
 \right]\in \r{ F}(\c{A})$.  
Let $\alpha_{n}$ be an automorphism 
from ${\mathcal A}_{n}$ onto ${\mathcal A}_{n}$ 
such that $\alpha_{n}((a,b))=(b,a)$.  
Since $\psi_{n}\circ \alpha_{n}=\alpha_{n+1}\circ \psi_{n}$ for any $n\in \mathbb{N}$, 
there exists an automorphism $\alpha:{\mathcal A}\rightarrow {\mathcal A}$ 
such that $\alpha|_{{\mathcal A}_{n}}=\alpha_{n}$.  
Therefore $\varphi_{2}=\varphi_{1}\circ \alpha$ 
and $\varphi_{1}=\varphi_{2}\circ \alpha$.  
On the other hand, 
$\r{F}({\mathcal A})\subset \Set{\left[ 
 \begin{array}{cc}
 2^{n}& 0 \\
 0& 2^{m} \\
 \end{array} 
 \right], 
 \left[ 
 \begin{array}{cc}
 0& 2^{n} \\
 2^{m}& 0 \\
 \end{array} 
 \right](n,m\in \m{Z})}$ by \ref{pro:Krelation} and \ref{lem:examplee}.  
We next show that if $\left[ 
 \begin{array}{cc}
 a& 0 \\
 0& b \\
 \end{array} 
 \right]\in \r{ F}(\c{A})$, then $a=b$.  
Let $\c{E}$ be an $\c{A}$-$\c{A}$ imprimitivity bimodule 
such that $R_{\c{A}}([\c{E}])\varphi_{1}=a\varphi_{1}$ and $R_{\c{A}}([\c{E}])\varphi_{2}=b\varphi_{2}$.  
Then $\sum^{n}_{i=1}\varphi_{1}(\braket{\xi_{i},\xi_{i}}_{\c{A}})=a$ 
and $\sum^{n}_{i=1}\varphi_{2}(\braket{\xi_{i},\xi_{i}}_{\c{A}})=b$, 
where $\set{\xi_{i}}^{n}_{i=1}$ is a basis of $\c{E}$.  
Put a projection $p=(\braket{\xi_{i},\xi_{j}}_{\c{A}})_{ij}$ on $M_{n}(\c{A})$.  
Then $\r{Tr}_{n}\otimes \varphi_{1}(p)=a$ and $\r{Tr}_{n}\otimes \varphi_{2}(p)=b$.  
Since $M_{n}(\c{A})=\lim_{i\rightarrow \infty}M_{n}(\c{A}_{i})$, 
there exist a projection $p_{0}$ in $M_{n}(\c{A}_{i})$ 
such that $\r{Tr}_{n}\otimes \varphi_{1}|_{\c{A}_{i}}
(p_{0})=a$ and $\r{Tr}_{n}\otimes \varphi_{2}|_{\c{A}_{i}}(p_{0})=b$.  
Put $p_{0}=p^{1}_{0}\oplus p^{2}_{0}$, 
where $p^{1}_{0}$ and $p^{2}_{0}$ are projections of $M_{n\Pi^{i}_{k=1}4p^{2}_{k}}(\m{C})$.  
Then $a=\r{Tr}_{n}\otimes (\cfrac{1}{2}+ \cfrac{3\Pi^{i+1}_{k=1} p_{k}^{2}}{(\Pi^{i+1}_{k=1}(p_{k}^{2}-1))\pi^{2}})\cdot \cfrac{1}{\Pi^{i}_{k=1}(4p_{k}^{2})}\r{Tr}_{\Pi^{i}_{k=1}(4p_{k}^{2})}(p_{1})+\r{Tr}_{n}\otimes (\cfrac{1}{2}- \cfrac{3\Pi^{i+1}_{k=1} p_{k}^{2}}{(\Pi^{i}_{k=1}(p_{k}^{2}-1))\pi^{2}})\cdot \cfrac{1}{\Pi^{i}_{k=1}(4p_{k}^{2})}\r{Tr}_{\Pi^{i}_{k=1}(4p_{k}^{2})}(p_{2})$ and $b=\r{Tr}_{n}\otimes (\cfrac{1}{2}- \cfrac{3\Pi^{i+1}_{k=1} p_{k}^{2}}{(\Pi^{i}_{k=1}(p_{k}^{2}-1))\pi^{2}})\cdot \cfrac{1}{\Pi^{i}_{k=1}(4p_{k}^{2})}\r{Tr}_{\Pi^{i}_{k=1}(4p_{k}^{2})}(p_{1})+\r{Tr}_{n}\otimes (\cfrac{1}{2}+ \cfrac{3\Pi^{i+1}_{k=1} p_{k}^{2}}{(\Pi^{i}_{k=1}(p_{k}^{2}-1))\pi^{2}})\cdot \cfrac{1}{\Pi^{i}_{k=1}(4p_{k}^{2})}\r{Tr}_{\Pi^{i}_{k=1}(4p_{k}^{2})}(p_{2})$.  
Let $q_{1},q_{2}$ be rational numbers.  
Since $a,b$ are rational numbers, if $(\cfrac{1}{2}+ \cfrac{3\Pi^{n}_{k=1} p_{k}^{2}}{(\Pi^{n}_{k=1}(p_{k}^{2}-1))\pi^{2}})q_{1}+(\cfrac{1}{2}- \cfrac{3\Pi^{n}_{k=1} p_{k}^{2}}{(\Pi^{n}_{k=1}(p_{k}^{2}-1))\pi^{2}})q_{2}=a$ 
and $(\cfrac{1}{2}- \cfrac{3\Pi^{n}_{k=1} p_{k}^{2}}{(\Pi^{n}_{k=1}(p_{k}^{2}-1))\pi^{2}})q_{1}+(\cfrac{1}{2}- \cfrac{3\Pi^{n}_{k=1} p_{k}^{2}}{(\Pi^{n}_{k=1}(p_{k}^{2}-1))\pi^{2}})q_{2}=b$, 
then $a=b$.  
Therefore $a=b$.  
Finally, we show $\left[ 
 \begin{array}{cc}
 2& 0 \\
 0& 2 \\
 \end{array} 
 \right]\in \r{ F}(\c{A})$.  
Let ${\mathcal B}_{n}=M_{\Pi^{n}_{k=1}(2p_{k}^{2})}\oplus M_{\Pi^{n}_{k=1}(2p_{k}^{2})}$.  
We define ${}^*$-homomorphisms $\phi_{n} :A_{n}\rightarrow A_{n+1}$ \\
by $\phi_{n}((a,b))=(\r{diag}(a,a,\cdots,a,b),\r{diag}(b,b,\cdots,b,a))$.  
Let ${\mathcal B}$ be the inductive limit of the sequence
\[\begin{CD}
      {\mathcal B}_{1} @>\phi_{1}>> {\mathcal B}_{2} @>\phi_{2}>> {\mathcal B}_{3} @>>> \cdots  \end{CD}. \]
Then $\c{A}$ is isomorphic to $M_{2^{\infty}}\otimes B$. 
Therefore $\c{A}$ is isomorphic to $M_{2}(\c{A})$.  

\end{proof}

\begin{ex}\label{ex:prime2}
Let $p>2$ be a prime number.  
Put ${\mathcal B}_{n}=M_{\Pi^{n}_{k=1}p^{k-1}}(\m{C})\oplus M_{\Pi^{n}_{k=1}p^{k-1}}(\m{C})$.  
We define ${}^*$-homomorphisms $\psi_{n} :B_{n}\rightarrow B_{n+1}$\\
by $\psi_{n}((a,b))=(\r{diag}(a,\cdots,a,b,\cdots,b),\r{diag}(b,\cdots,b,a,\cdots,a))$, 
where the multiplicities of $a$ are $p^{n}-2^{n-1}$ and $2^{n-1}$ respectively and 
the one of $b$ are $2^{n-1}$ and $p^{n}-2^{n-1}$ respectively.  
Let ${\mathcal B}$ be the inductive limit of the sequence
\[\begin{CD}
      {\mathcal B}_{1} @>\psi_{1}>> {\mathcal B}_{2} @>\psi_{2}>> {\mathcal B}_{3} @>>> \cdots  \end{CD}. \]
Put $\c{A}=M_{p^{\infty}}\otimes\c{B}$.  
Then $\c{A}$ is simple and dim$\r{ T}(\c{A})=2$.  
Since $\frac{2}{p}$ is an algebraic number, \\ 
$\Pi^{\infty}_{n=1}(1-(\frac{2}{p})^{n})$ is a transcendental number.  
As in the same proof of the previous example, \\
$\r{F}({\mathcal A})= \Set{\left[ 
 \begin{array}{cc}
 p^{n}& 0 \\
 0& p^{n} \\
 \end{array} 
 \right], 
 \left[ 
 \begin{array}{cc}
 0& p^{n} \\
 p^{n}& 0 \\
 \end{array} 
 \right](n\in \m{Z})}$ and $\r{F}_{\r{det}}(\c{A})=\set{p^{2n}: n\in \m{Z}}$.  
\end{ex}

\begin{ex}\label{ex:matrix}
Let $\{n_{i}\}^{m}_{i=1}$ be a finite subset of $\mathbb{N}$ 
such that $1\leq n_{1}$ and $n_{i}\leq n_{i+1}$ for any $i$.  
Put ${\mathcal B}=\oplus ^{m}_{i=1}M_{n_{i}}(\mathbb{C})$.   
We denote by $\varphi_{i}$ the normalized trace on $M_{n(i)}(\mathbb{C})$.  
We show 
\begin{eqnarray}
\r{ F}({\mathcal B})=\{ DU(\sigma) : D_{ii}=\cfrac{n(i)}{n(\sigma(i))}\ \sigma\in S_{N}  \}
\end{eqnarray}
Put ${\mathcal A}=\mathbb{C}^{m}$. 
The $C^*$-algebras ${\mathcal A}$ and ${\mathcal B}$ are Morita equivalent 
and $\oplus ^{m}_{i=1}M_{1,n_{i}}(\mathbb{C})$, 
we write by ${\mathcal F}$, 
is an ${\mathcal A}$-${\mathcal B}$ imprimitivity bimodule.  
Then the representation matrix $P$ of $R_{{\mathcal A}{\mathcal B}}({\mathcal F})$ is a diagonal matrix 
and $P_{ii}=n(i)$.  
By \ref{pro:morita}, $\r{ F}({\mathcal B})
=R_{{\mathcal B}{\mathcal A}}({\mathcal F}^*)\{U(\sigma): \sigma \in S_{N}\}R_{{\mathcal A}{\mathcal B}}({\mathcal F})$.  
Especially, if $\c{A}=\m{C}^{2}$ and ${\mathcal B}=M_{2}(\mathbb{C})\oplus M_{3}(\mathbb{C})$, \\
then $\r{ F}(\c{A})=\Set{\left[ 
 \begin{array}{cc}
 1& 0 \\
 0& 1 \\
 \end{array} 
 \right], 
 \left[ 
 \begin{array}{cc}
 0& 1 \\
 1& 0 \\
 \end{array} 
 \right]
}$, \\
$\r{ F}({\mathcal B})=\left[ 
 \begin{array}{cc}
 \cfrac{1}{2}& 0 \\
 0&  \cfrac{1}{3} \\
 \end{array} 
 \right]\r{ F}(\c{A})\left[ 
 \begin{array}{cc}
 2& 0 \\
 0& 3 \\
 \end{array} 
 \right]
=\Set{\left[ 
 \begin{array}{cc}
 1& 0 \\
 0& 1 \\
 \end{array} 
 \right], 
 \left[ 
 \begin{array}{cc}
 0& \cfrac{3}{2} \\
 \cfrac{2}{3}& 0 \\
 \end{array} 
 \right]
}$ and \\
$\r{ F}_{\r{det}}(\c{A})=\r{ F}_{\r{det}}(\c{B})=\{1\}$.
\end{ex}

\begin{ex}\label{ex:irr}
Let $\theta$ be an irrational number. 
The irrational rotation algebra ${\mathcal A}_{\theta }$ 
has a unique normalized trace $\varphi_{\theta }$.  
The irrational rotation algebras ${\mathcal A}_{\theta },{\mathcal A}_{\eta }$ 
are Morita equivalent if and only if $\eta=\cfrac{a\theta +b}{c\theta +d}$,
$\left[ 
 \begin{array}{cc}
 a& b \\
 c& d \\
 \end{array} 
 \right] \in$ GL${}_{2}(\mathbb{Z})$ 
where GL${}_{2}(\mathbb{Z})$ is the set of integer-valued $2\times 2$ matrices $A$ 
which satisfy $\r{det}A=1$ or $-1$.  
If $X$ be an ${\mathcal A}_{\eta}$-${\mathcal A}_{\theta }$ imprimitivity bimodule 
with ${\mathcal A}_{\eta}$-valued left inner product ${}_{{\mathcal A}_{\eta}}\langle\cdot ,\cdot \rangle$ 
and with ${\mathcal A}_{\theta }$-valued right inner product $\langle\cdot ,\cdot \rangle_{{\mathcal A}_{\theta }}$, 
we have the following equation; $\varphi_{\theta }(\langle \xi ,\zeta \rangle_{{\mathcal A}_{\theta }})=|c\theta +d|\varphi_{\eta  }({}_{{\mathcal A}_{\eta }}\langle \zeta ,\xi \rangle)$.  
These facts can be found in \cite{Rie1} and \cite{Rie3}.  
Put ${\mathcal B}={\mathcal A}_{\theta }\oplus {\mathcal A}_{\eta }$ and ${\mathcal C}={\mathcal A}_{\theta }\oplus {\mathcal A}_{\theta }$.  
Since ${\mathcal A}_{\theta }\oplus X$ is a ${\mathcal B}$-${\mathcal C}$ imiprimitivity bimodule, $\r{ F}({\mathcal B})=
\left[ 
 \begin{array}{cc}
 1& 0 \\
 0& |c\theta +d| \\
 \end{array} 
 \right]\r{ F}({\mathcal C})
\left[ 
 \begin{array}{cc}
 1& 0 \\
 0& \cfrac{1}{|c\theta +d|} \\
 \end{array} 
 \right] 
$ and $\r{ F}_{\r{det}}(\c{B})=\r{ F}_{\r{det}}(\c{C})$ by the equation.  
From \ref{cor:dsum} and corollary 3.18 of \cite{NY}, we can calculate $\r{ F}({\mathcal B})$ explicitly.  \\
Especially let $\theta =\sqrt[]{\mathstrut 5}$ and $\eta=\cfrac{1}{\sqrt[]{\mathstrut 5}}$, \\
then $\r{ F}({\mathcal B})=\{
\left[ 
 \begin{array}{cc}
 (\sqrt[]{\mathstrut 5}+2)^{n}& 0 \\
 0& (\sqrt[]{\mathstrut 5}+2)^{m} \\
 \end{array} 
 \right], 
\left[ 
 \begin{array}{cc}
 0& \cfrac{1}{\sqrt[]{\mathstrut 5}}(\sqrt[]{\mathstrut 5}+2)^{n} \\
\sqrt[]{\mathstrut 5}(\sqrt[]{\mathstrut 5}+2)^{m} & 0 \\
 \end{array} 
 \right] \\
: n,m \in \mathbb{Z} \}$ and $\r{ F}_{\r{det}}(\c{B})=\set{(\sqrt[]{\mathstrut 5}+2)^{n}:n\in \m{Z}}$.  
Let $\theta $ be a non-quadratic number and $\eta=\cfrac{a\theta +b}{c\theta +d}$, 
then $\r{ F}({\mathcal B})=\{
\left[ 
 \begin{array}{cc}
 1& 0 \\
 0& 1 \\
 \end{array} 
 \right], 
\left[ 
 \begin{array}{cc}
 0& \cfrac{1}{|c\theta +d|} \\
 |c\theta  +d|& 0 \\
 \end{array} 
 \right]\}$ and $\r{ F}_{\r{det}}(\c{B})=\{1\}$ .
\end{ex}

Put $G_{j}=\set{\cfrac{m}{\Pi^{k}_{i=1}p^{j}_{i}}:m\in\m{Z},\ k\in\m{N}}$.  
Then $G_{j}$ is an additive dense subgroup of $\m{R}$.  

\begin{ex}
Let $n$ be a natural number.  \\
Put $G=\Pi^{n}_{l=1}G_{l}$ and $G^{+}=\set{(g_{1},g_{2},\cdots,g_{n}): g_{i}>0 }\cup {0}$, where $0$ is the additive unit of $\m{R}^{n}$.  
Then $(G,G^{+})$ is unperforated and has the Riesz interpolation property.  
We denote by $\c{A}_{unit}$ the unital simple $AF$-algebra 
the triples of which are isomorphic to $(G,G^{+},(1,1,\cdots,1))$.  
Then $\r{ F}(\c{A}_{unit})=\set{I_{n}}$ and $\r{ F}_{\r{det}}(\c{A}_{unit})=\set{1}$, where $\r{I}_{n}$ is a unit of $M_{n}(\m{C})$.  
\end{ex}

\begin{ex}
Let $n$ be a natural number.  \\
Put $G=\Pi^{n}_{l=1}G_{1}$ and $G^{+}=\set{(g_{1},g_{2},\cdots,g_{n}): g_{i}>0 }\cup {0}$, where $0$ is the additive unit of $\m{R}^{n}$.  
Then $(G,G^{+})$ is unperforated and has the Riesz interpolation property.  
We denote by $\c{A}_{S_{n}}$ the unital simple $AF$-algebra 
the triples of which are isomorphic to $(G,G^{+},(1,1,\cdots,1))$.  
Then $\r{ F}(\c{A}_{S_{n}})=U(S_{n})$.  
\end{ex}
\begin{ex}
Let $\theta$ be a non-quadratic number.  
Put $G_{\theta}=(G_{1}+\theta G_{1})\oplus (G_{1}+\theta G_{1})$ and $G^{+}_{\theta}=\set{(g,h):g>0,h>0}\cup {0}$.  
Then $(G_{\theta},G^{+}_{\theta})$ is unperforated and has the Riesz interpolation property.  
Moreover, $(1,1)$ and $(1,\theta)$ are order units of $(G_{\theta},G^{+}_{\theta})$ .  
We denote by $\c{A}_{(1,1)}$ and by $\c{A}_{(1,\theta)}$ the unital simple $AF$-algebras 
the triples of which are isomorphic to $(G_{\theta},G^{+}_{\theta},(1,1))$ and $(G_{\theta},G^{+}_{\theta},(1,\theta))$ respectively.  
Then $\c{A}_{(1,1)}$ and $\c{A}_{(1,\theta)}$ are Morita equivalent.  
Moreover, $\r{ F}(\c{A}_{(1,1)})=\Set{\left[ 
 \begin{array}{cc}
 1& 0 \\
 0& 1 \\
 \end{array} 
 \right], 
 \left[ 
 \begin{array}{cc}
 0& 1 \\
 1& 0 \\
 \end{array} 
 \right]
}$, $\r{ F}(\c{A}_{(1,\theta)})=\{
\left[ 
 \begin{array}{cc}
 1& 0 \\
 0& 1 \\
 \end{array} 
 \right], 
\left[ 
 \begin{array}{cc}
 0& \theta \\
 \cfrac{1}{\theta}& 0 \\
 \end{array} 
 \right]
\}
$ and $\r{ F}_{\r{det}}(\c{A}_{(1,1)})=\r{ F}_{\r{det}}(\c{A}_{(1,\theta)})=\set{1}$.  
\end{ex}

\begin{ex}\label{ex:irratio2}
Let $\theta$ be a quadratic number.  
Put $\bar G_{\theta}=(G_{1}+\theta G_{1}+\pi G_{1})\oplus (G_{1}+\theta G_{1}+\pi G_{1})$ and $\bar G^{+}_{\theta}=\set{(g,h):g>0,h>0}\cup {0}$.  
Then $(\bar G_{\theta},\bar G^{+}_{\theta})$ is unperforated and has the Riesz interpolation property.  
Moreover, $(1,1)$ and $(1,\theta)$ are order units of $(\bar G_{\theta},\bar G^{+}_{\theta})$.  
We denote by $\c{B}_{(1,1)}$ and by $\c{B}_{(1,\theta)}$ the unital simple $AF$-algebras 
the triples of which are isomorphic to $(\bar G_{\theta},\bar G^{+}_{\theta},(1,1))$ and $(\bar G_{\theta},\bar G^{+}_{\theta},(1,\theta))$ respectively.  
Then $\c{B}_{(1,1)}$ and $\c{B}_{(1,\theta)}$ are Morita equivalent.  
Moreover, $\r{ F}(\c{B}_{(1,1)})=\Set{\left[ 
 \begin{array}{cc}
 1& 0 \\
 0& 1 \\
 \end{array} 
 \right], 
 \left[ 
 \begin{array}{cc}
 0& 1 \\
 1& 0 \\
 \end{array} 
 \right]
}$, $\r{ F}(\c{B}_{(1,\theta)})=\{
\left[ 
 \begin{array}{cc}
 1& 0 \\
 0& 1 \\
 \end{array} 
 \right], 
\left[ 
 \begin{array}{cc}
 0& \theta \\
 \cfrac{1}{\theta}& 0 \\
 \end{array} 
 \right]
\}
$ and $\r{ F}_{\r{det}}(\c{A}_{(1,1)})=\r{ F}_{\r{det}}(\c{A}_{(1,\theta)})=\set{1}$.  
\end{ex} 

Hence next proposition follows.

\begin{thm}\label{thm:Z2}
Let $G$ be a subgroup of $GL_{2}(\m{R})$ whose elements are represented as $DU(\sigma)$, 
where $D$ is a positive diagonal matrix and $\sigma\in S_{n}$.  
If $G$ is isomorphic to $\m{Z}_{2}$ as a group, 
then there exists a simple $AF$-algebra $\c{A}$ such that $\r{ F}(\c{A})=G$.  
\end{thm}

\begin{ex}\label{ex:zid}
Let $n$ be a natural number.  
Let $\theta_{1},\cdots,\theta_{n}$ be non-quadratic numbers which satisfy $\theta_{i}\neq\cfrac{a+b\theta_{j}}{c+d\theta_{j}}$ for any $a,b,c,d$ in $\m{Z}$ and for any $i,j$.  
Put $G_{\set{\theta_{1},\cdots,\theta_{n}}}=\oplus^{n}_{i=1}(\m{Z}+\theta_{i}\m{Z})$ and 
$G^{+}_{\set{\theta_{1},\cdots,\theta_{n}}}=\set{(g_{1},\cdots,g_{n}):g_{i}>0}\cup \set{0}$.  \\
Then $(G_{\set{\theta_{1},\cdots,\theta_{n}}},G^{+}_{\set{\theta_{1},\cdots,\theta_{n}}})$ is unperforated and has the Riesz interpolation property.  
Moreover, $(1,\cdots,1)$ is an order unit of $(G_{\set{\theta_{1},\cdots,\theta_{n}}},G^{+}_{\set{\theta_{1},\cdots,\theta_{n}}})$.  
We denote by $\c{A}_{\set{\theta_{1},\cdots,\theta_{n}}}$ the unital simple $AF$-algebra 
the triple of which is isomorphic to $(G_{\set{\theta_{1},\cdots,\theta_{n}}},G^{+}_{\set{\theta_{1},\cdots,\theta_{n}}},(1,\cdots,1))$.  
Then $\r{ F}(\c{A}_{\set{\theta_{1},\cdots,\theta_{n}}})=\set{\r{I}_{n}}$.  
\end{ex}

By using \ref{ex:zid}, we show the followimg proposition.  

\begin{thm}
For any natural number $n$, 
there exist uncountably many mutually nonisomorphic simple (non)nuclear unital $C^{*}$-algebras $\c{A}$ with $n$-dimensional trace space such that $\r{ F}(\c{A})=\set{I_{n}}$.  
\end{thm}

\begin{proof}
We show the case $n=2$.  
If $\c{A}_{\set{\theta_{1},\theta_{2}}}$ is isomorphic to $\c{A}_{\set{\theta^{\prime}_{1},\theta^{\prime}_{2}}}$, then there exist $m_{1},m_{2}\in\m{Z}$ such that $\set{\theta^{\prime}_{1},\theta^{\prime}_{2}}=\set{\theta_{1}+m_{1},\theta_{2}+m_{2}}$.  
Since there exist uncountably many sets $\set{\theta_{1},\theta_{2}}$ each pair of which do not have that relation, there exist uncountably many mutually nonisomorphic simple unital $AF$-algebras.  
In the case of nonnuclear, all one have to do is to consider $\c{A}_{\set{\theta_{1},\theta_{2}}}\otimes C^{*}_{r}(\m{F}_{2})$.    
\end{proof}

\begin{pro}\label{pro:det}
Let $G$ be a countable subgroup of $\m{R}^{\times }_{+}$ and let $n$ be a natural number.  
Then there exist uncountably many mutually nonisomorphic separable simple nonnuclear unital $C^{*}$-algebras $\c{A}$ with $n$-dimensional trace space such that $\r{ F}_{\r{det}}(\c{A})\supset G$.  
\end{pro}
\begin{proof}
Let $r \not\in G$ be a real number of $\m{R}^{\times }_{+}$.  
We denote $G_{r}$ the subgroup of $\m{R}^{\times }_{+}$ generated by $r$ and $G$.   
Then there exist a separable non-nuclear unital $C^{*}$-algebra $\c{B}$ with unique trace 
such that $\r{ F}(\c{B}_{r})=\set{g^{\frac{1}{n}}:g\in G_{r}}$.  
Then $\c{B}_{r}\otimes \c{A}_{S_{n}}$ is a separable non-nuclear unital $C^{*}$-algebra with  $n$-dimensional trace space such that $\r{ F}_{\r{det}}(\c{B}\otimes \c{A}_{S_{n}})\supseteq G_{r}\supsetneq G$.  
Since $r$ is arbitrary and $\r{ F}_{\r{det}}(\c{B}\otimes \c{A}_{S_{n}})$ is countable,  
there exist uncountably many mutually nonisomorphic separable simple nonnuclear unital $C^{*}$-algebras.  
\end{proof}
\begin{cor}
For any natural number $n$, there exist uncountably many mutually nonisomorphic separable simple nonnuclear unital $C^{*}$-algebras $\c{A}$ with $n$-dimensional trace space each of which fundamental groups $\r{ F}(\c{A})$ are different.  
\end{cor}

\section{Exact sequence of Picard group and fundamental group}\label{sec:exact}
 In this section, we shall show the diagram with respect to fundamental groups $\r{ F}(\c{A})$
and Picard groups of $C^*$-algebra $\c{A}$.  This construction generalizes the proposition 3.26 of \cite{NY}.  
\begin{Def}\label{def;tracepreserve}
Let ${\mathcal A}$ be a unital $C^*$-algebra.  
We denote by $\mathrm{Aut}_{\r{ T}({\mathcal A})}({\mathcal A})$ the set of automorphisms which are trace invariant.  
Then $\mathrm{Int}({\mathcal A})$ is a normal subgroup of $\mathrm{Aut}_{\r{ T}({\mathcal A})}({\mathcal A})$.  
We denote by $\mathrm{Out}_{\r{ T}({\mathcal A})}({\mathcal A})$ the quotient group $\mathrm{Aut}_{\r{ T}({\mathcal A})}({\mathcal A})/\mathrm{Int}({\mathcal A})$.
\end{Def}
Because $\mathrm{Aut}_{\r{ T}({\mathcal A})}({\mathcal A})$ is normal subgroup of $\mathrm{Aut}({\mathcal A})$, 
$\mathrm{Out}_{\r{ T}({\mathcal A})}({\mathcal A})$ is a normal sugroup of $\mathrm{Out}({\mathcal A})$.  
We denote by $\rho_{{\mathcal A}}|$ the restriction of $\rho_{{\mathcal A}}$ 
to $\mathrm{Aut}_{\r{ T}({\mathcal A})}({\mathcal A})$ 
and denote by $S_{\c{A}}$ the representation of $\r{F}^{tr}(\c{A})$ into $\c{L}(\r(T)(\c{A}))$.  
We say that $\r{ T}({\mathcal A})$ separates equivalence classes of projections 
if for any fixed natural number $n$ and for any fixed projecition $p, q\in M_{n}({\mathcal A})$, 
$\r{Tr}^{(n)}\otimes\varphi (p)=\r{Tr}^{(n)}\otimes\varphi(q)$ 
for all $\varphi$ in $\r{ T}({\mathcal A})$, 
then $p$ and $q$ are von Neumann equivalent.
\begin{thm}\label{thm:exacts}
Let $A$ be a unital $C^*$-algebra with finite dimensional bounded trace space.  
If $\r{ T}({\mathcal A})$ separates equivalence classes of projections, 
then we have the following commutative diagram whose horizontal lines are exact.  
\[\begin{CD}
{1} @>>> \mathrm{Out}_{\r{ T}({\mathcal A})}({\mathcal A}) @>\rho_{\mathcal A}|>> \mathrm{Pic}({\mathcal A}) @>S_{{\mathcal A}}R_{{\mathcal A}}>> \r{ F}({\mathcal A})
 @>>> {1}\\
@. @AidAA @A\rho_{\mathcal A}AA @Ai_{F}AA\\
{1} @>>> \mathrm{Out}_{\r{ T}({\mathcal A})}({\mathcal A}) @>i_{out}>> \mathrm{Out}({\mathcal A}) @>\rho_{\mathcal A}S_{{\mathcal A}}R_{{\mathcal A}}>> U(S_{n})\cap \r{ F}({\mathcal A})
 @>>> {1} \end{CD}.   \]  
In this diagram, $i_{out}$ and $i_{F}$ are the inclusion maps from $\mathrm{Out}_{\r{ T}({\mathcal A})}({\mathcal A})$ into $\mathrm{Out}({\mathcal A})$ and from $U(S_{n})\cap \r{ F}({\mathcal A})$ into $\r{ F}(\c{A})$ respectively.  
\end{thm}
\begin{proof}
It is sufficient to show horizontal lines are exact.  
We show the first line is exact.  
The map $\rho_{\mathcal A}$ is one-to-one and  $\mathrm{Im}\rho_{\mathcal A}\subset \mathrm{Ker}(S_{{\mathcal A}}R_{{\mathcal A}})$ by the chapter 2 and  
$S_{A}R_{A}$ is onto by definition.  
We shall show that $\mathrm{Ker}(S_{A}R_{A})\subset \mathrm{Im}\rho_{\mathcal A}$.   
Let $[{\mathcal E}]$ be in $\mathrm{Ker}(S_{{\mathcal A}}R_{{\mathcal A}})$, 
let $\{\xi_{i}\}_{i=1}^{k}$ be the basis of ${\mathcal E}$, 
let $p=(\langle \xi_{i}, \xi_{j} \rangle_{\mathcal A})_{ij}$ in $M_{k}({\mathcal A})$, 
and let $\Phi$ be the isomorphism from ${\mathcal A}$ to $pM_{n}({\mathcal A})p$.  
Because $S_{{\mathcal A}}R_{{\mathcal A}}([{\mathcal E}])=id_{\r{ T}({\mathcal A})}$, 
$\sum_{i=1}^{k}\varphi_{j}(\langle \xi_{i}, 
a\xi_{i} \rangle_{{\mathcal A}})=\varphi_{j}(a)$ 
for all $\varphi_{j}$ in $\partial_{e}\r{T}(\c{A})^{+}_{1}$ 
and for all $a$ in $A$.  
Substituting $a=1_{{\mathcal A}}$ into the previous formula, 
we can obtain the formula 
that $\r{Tr}^{(k)}\otimes\varphi(p)=\r{Tr}^{(k)}\otimes\varphi(1_{\c{A}}\otimes e_{11})=1$.  
By assumption, there exists a partial isometry $w$ 
such that $p=w^*w$ and $1\otimes e_{11}=ww^*$ in $M_{k}({\mathcal A})$.  
Then there exists an automorphism $\alpha$ of ${\mathcal A}$ 
such that $w\Phi(a)w^*=\alpha(a)\otimes e_{11}$.  
Hence $[{\mathcal E}]=[{\mathcal E}_{\alpha}]$.  
Since $[{\mathcal E}]$ is in $\mathrm{Ker}(S_{{\mathcal A}}R_{{\mathcal A}})$, 
$\alpha$ is an element of $\mathrm{Out}_{\r{ T}({\mathcal A})}({\mathcal A})$.  
We next show that $\rho_{\mathcal A}S_{{\mathcal A}}R_{{\mathcal A}}$ is onto.  
Since $S_{{\mathcal A}}R_{{\mathcal A}}$ is onto, 
for all $U(\sigma)\in U(S_{n})\cap \r{ F}({\mathcal A})$ 
there exists a imprimitivity bimodule ${\mathcal E}$ 
such that $S_{\mathcal A}R_{\mathcal A}([{\mathcal E}])=U(\sigma)$.  
Let $\{\xi_{i}\}_{i=1}^{l}$ be a basis of ${\mathcal E}$.  
Then $\sum_{i=1}^{l}\varphi_{j}(\langle \xi_{i}, 
a\xi_{i} \rangle_{{\mathcal A}})=\varphi_{\sigma(j)}(a)$.  
As in the same proof in \ref{thm:exacts}, 
there exists an automorphism $\alpha$ such that $[{\mathcal E}]=[{\mathcal E}_{\alpha}]$.
\end{proof}

By using the above diagram, we have the following corollary.    

\begin{cor}
Let $A$ be a unital $C^*$-algebra with finite dimensional bounded trace space.
Put $\r{ dim}\r{ T}({\mathcal A})=n$.    
Under the same assumption in \ref{thm:exacts}, 
if $\r{ F}(\c{A})\subset U(S_{n})$, then $\mathrm{Out}(\c{A})$ is isomorphic to $\r{ Pic}(\c{A})$.  
\end{cor}

We will show the relation between the scaling group and the fundamental group.  

\begin{Def}\label{def:db}
Let ${\mathcal A}$ be a unital $C^*$-algebra with finite dimensional bounded trace space
and let $\set{\varphi_{i}}^{n}_{i=1}$ be a basis of ${\rm T}(\c{A})$.  
A dual system of $\set{\varphi_{i}}^{n}_{i=1}$ in $\c{A}$ is the subset $\set{u_{i}}^{n}_{i=1}$ of $\c{A}$ satisfying $\varphi_{i}(u_{j})=\delta _{ij}$, 
where $\delta_{ij}$ is Kronecker's delta.  
\end{Def}
\begin{lem}\label{lem:dbexist}
Let ${\mathcal A}$ be a unital $C^*$-algebra with finite dimensional bounded trace space.  
There exists a dual system $\set{u_{i}}^{n}_{i=1}$ for any basis $\set{\varphi_{i}}^{n}_{i=1}$ of ${\rm T}(\c{A})$.  
\end{lem}

\begin{proof}
We define a linear map $\Phi:\c{A}\rightarrow \m{C}^{n}$ 
by $\Phi(a)=(\varphi_{1}(a),\varphi_{2}(a),\cdot \cdot \cdot ,\varphi_{n}(a))$.  
It is sufficient to show ${\rm dim}\mathrm{Im}(\Phi)=n$.  
Obviously,  ${\rm dim}\mathrm{Im}(\Phi)\leq n$.  
We will show ${\rm dim}\mathrm{Im}(\Phi)={\rm dim}(\c{A}/\mathrm{Ker}(\Phi))\geq n$.  
The linear map $\iota :{\rm T}(\c{A})\rightarrow (\c{A}/\mathrm{Ker}(\Phi))^{*}$ 
given by $\iota (\varphi)([a])=\varphi(a)$ is welldefined and injective, 
where $[a]$ is an equivalent class of $a$ in $\c{A}$.   
Therefore ${\rm dim}(\c{A}/\mathrm{Ker}(\Phi))\geq n$.  
\end{proof}

\begin{rem}
Let $\c{A}$ be a simple $C^{*}$-algebra 
and let $\tau$ be a non-zero bounded trace on $\c{A}$.  
Since $\tau$ is a trace, $I=\set{a\in\c{A}:\tau(a^{*}a)=0}$ is an closed two-sided ideal.  
Therefore $I=\set{0}$ because $\c{A}$ is simple.  
We suppose ${\rm dim}{\rm T}(\c{A})\geq 2$.  
Then no elements of dual system of ${\rm T}(\c{A})$ is positive 
because $\tau(a)>0$ for any $\tau$ in ${\rm T}(\c{A})^{+}$ and for any positive element $a$ in $\c{A}$.  
\end{rem}  

By this dual basis, we can see 
\[\r{F}(\c{A})=\set{((tr_{k}\otimes \varphi_{i})\circ \Phi (u_{j}))_{ij}\mid (p,\Phi ): {\rm s.s.p},\ p\in M_{k}(\c{A})}\].   

Let $\c{A}$ be a unital $C^{*}$-algebra with finite dimensional tracial space and with no unbounded trace.  
We suppose $\sharp(\partial_{e}(\r{T}(\c{A}))=n$.  
Say $\set{\varphi_{i}}^{n}_{i=1}=\partial_{e}(\r{T}(\c{A}))$.  
Let $\set{u_{j}}^{n}_{j=1}$ be a dual basis of $\set{\varphi_{i}}^{n}_{i=1}$.  
If $\alpha\in \r{Aut}(\c{A}\otimes\m{K})$, then $(\tau\otimes \r{Tr})\circ \alpha$ is a densely defined, lower semicontinuous trace 
for any $\tau\in \r{T}(\c{A})$.  
Since $\c{A}$ has no unbounded trace,  we can write $(\tau\otimes \r{Tr})\circ \alpha$ as a linear combination of $\set{\varphi_{i}\otimes \r{Tr}}^{n}_{i=1}$.  
Put
\[\c{S}(\c{A})=\set{((\varphi_{i}\otimes \r{Tr})\circ \alpha(u_{j}\otimes e_{11}))_{ij} \mid \alpha \in \r{Aut}(\c{A}\otimes\m{K}) }.\]

Even if $\c{A}\otimes \m{K}$ is isomorphic to $\c{B}\otimes \m{K}$, $S(\c{A})$ may not be $S(\c{B})$ because $S(\c{A})$ depends on $\partial_{e}(\r{T}(\c{A}))$ (See \ref{ex:scale}).     
We shall show that this set is equal to the fundamental group of $\c{A}$.  

\begin{pro}
Let $\c{A}$ be a unital $C^{*}$-algebra with finite dimensional tracial space and with no unbounded trace.   
Then $\c{S}(\c{A})=\r{F}(\c{A})$.  
\end{pro}

\begin{proof}
Let $p$ be a self-similar full projection and let $\Phi$ be an isomorphism from $\c{A}\rightarrow pM_{n}(\c{A})p$.  
Then $\Phi\otimes \r{id}_{\m{K}}:\c{A}\otimes \m{K}\rightarrow pM_{n}(\c{A})p\otimes \m{K}=(p\otimes I)M_{n}(\c{A})\otimes \m{K}(p\otimes I)$ is an isomorphism.  
Since $p$ is full, there exists a partial isometry $v\in \r{M}(M_{n}(\c{A})\otimes \m{K})$ such that $v^{*}v=p\otimes I$ and that $vv^{*}=1$ by lemma 2.5 on \cite{BR}.  
Put $\beta_{v}(x)=vxv^{*}$ for $x\in \r{M}(M_{n}(\c{A})\otimes \m{K})$ and $\alpha=(\psi_{n}\otimes \r{id}_{\c{A}})\circ\beta_{v}\circ(\Phi\otimes \r{id}_{\m{K}})$,  
where $\psi_{n}$ is an isomorphism from $M_{n}(\m{C})\otimes \m{K}$ onto $\m{K}$ and $\r{id}_{\c{A}}$ is an identity map on $\c{A}$.  
Then $\alpha$ is an automorphism on $\c{A}\otimes \m{K}$ and  $(\tau\otimes \r{Tr})\circ \alpha(a)=\tau\otimes \r{Tr}_{n}\circ \Phi(a)$ for any $a$ in $\c{A}$.  
Therefore  $\c{S}(\c{A})\supset \r{F}(\c{A})$.  
Conversely, let $\alpha$ be an automorphism on $\c{A}\otimes \m{K}$.  
Put $p=\alpha(1\otimes e_{11})$.  
Then there exists a projection $q$ in $M_{n}(\c{A})$ such that $p$ and $q$ are von Neumann equivalent.  
We define an isomorphism $\Psi:p(\c{A}\otimes \m{K})p\rightarrow q(M_{n}(\c{A}))q$ 
by $\Psi(a)=vav^{*}$, where $v$ is a partial isometry satisfying $v^{*}v=p$ and $vv^{*}=q$.  
Since $\Psi\circ \alpha$ is the isomorphism of the map from $\c{A}$ onto $q(M_{n}(\c{A}))q$, 
if $\Phi$ is the induced map by this,  $(\tau\otimes \r{Tr})\circ \alpha(a)=\tau\otimes \r{Tr}_{n}\circ \Phi(a)$ for any $a$ in $\c{A}$.  
Because $1\otimes e_{11}$ is full, $q$ is a full projection of $M_{n}(\c{A})$.   
Therefore  $\c{S}(\c{A})\subset \r{F}(\c{A})$.  
Hence $\c{S}(\c{A})=\r{F}(\c{A})$
\end{proof}

\begin{ex}\label{ex:scale}
Put $\c{A}=M_{2}(\m{C})\oplus M_{3}(\m{C})$ and $\c{B}=\m{C}\oplus \m{C}$. \\ 
Then $S(\c{A})=\r{F}(\c{A})=\Set{\left[ 
 \begin{array}{cc}
 1& 0 \\
 0& 1 \\
 \end{array} 
 \right], 
 \left[ 
 \begin{array}{cc}
 0& \cfrac{3}{2} \\
 \cfrac{2}{3}& 0 \\
 \end{array} 
 \right]}
$ and \\
$S(\c{B})=\r{F}(\c{B})=\Set{\left[ 
 \begin{array}{cc}
 1& 0 \\
 0& 1 \\
 \end{array} 
 \right], 
 \left[ 
 \begin{array}{cc}
 0& 1 \\
 1& 0 \\
 \end{array} 
 \right]}$ by \ref{ex:matrix}.  
 Although $\c{A}$ and $\c{B}$ are Morita equivalent, $\r{F}(\c{A})\neq \r{F}(\c{B})$.  
\end{ex}

\end{document}